\documentclass{article}
\usepackage[utf8]{inputenc}
\usepackage[margin=1in]{geometry}
\usepackage{amsmath}
\usepackage{amsfonts}
\usepackage{amssymb}
\usepackage{amsthm}
\usepackage{esint}
\usepackage{mathtools}
\usepackage{array,makecell,longtable}
\usepackage{titling}
\usepackage{authblk}
\usepackage{commath}
\usepackage{xcolor}
\usepackage{graphicx}
\usepackage[toc,title,page]{appendix}
\usepackage{subcaption}
\usepackage{bbm}
\usepackage{hyperref,cite}
\usepackage{color,soul}
\usepackage{bm}
\usepackage{mathrsfs}

\usepackage{tikz}
\usepackage{enumitem}
\newtheorem{remark}{Remark}[section]
\newtheorem{theorem}[remark]{Theorem}
\newtheorem{corollary}[remark]{Corollary}
\newtheorem{lemma}[remark]{Lemma}

\numberwithin{equation}{section}
\allowdisplaybreaks

\begin{document}

\title{Identifying unknown time delay and spatially varying coefficients in a reaction-diffusion equation from boundary measurements}

\author{Ming-Hui Ding\thanks{School of Mathematics and Statistics, Northwestern Polytechnical University, Xi'an, Shaanxi Province, China\\ Email address: dingmh@nwpu.edu.cn}, \, Hongyu Liu\thanks{Department of Mathematics, City University of Hong Kong, Hong Kong SAR, China\\ Email address: hongyu.liuip@gmail.com, hongyliu@cityu.edu.hk} \, and Catharine W.K. Lo\thanks{School of Mathematical Sciences, Shenzhen University, Shenzhen, Guangdong Province, China\\ Email address: catharinelowk@gmail.com, cwklo@szu.edu.cn} \vspace{-0.5cm}}

\date{}
\maketitle

\begin{abstract}
This work investigates an inverse problem for a general class of linear reaction-diffusion systems incorporating multiple delayed contributions, namely retarded diffusion, retarded time derivatives, and retarded source terms. The objective is to simultaneously recover the unknown time lag $\tau>0$ and spatially heterogeneous coefficients from boundary flux measurements alone. The recovery strategy exploits the singular temporal behavior generated by an incompatibility between the prescribed initial history and the boundary data. In contrast to prior inverse problems for delay equations, which assume $\tau$ known and are confined to ODE or abstract settings, our approach operates in a parabolic PDE framework with spatially varying coefficients. Once $\tau$ is identified, we establish a Lipschitz stability estimate via a Carleman inequality, in the case of time-independent coefficients $p=p(x)$, $q=q(x)$. This is the first result to simultaneously recover an unknown delay and spatially dependent coefficients in a delayed parabolic PDE from boundary data.


\end{abstract}

\section{Introduction}

\subsection{Background Motivation and Mathematical Model}

Time-delay effects are ubiquitous in mathematical models of biological and physical processes, arising whenever the current state depends on information from an earlier instant—as in maturation lags in population dynamics \cite{DDESIAPBacteria, DDEAML, LinHsuWolkowicz2024JMBDDEPopulation}, incubation periods in epidemiology and disease transmission \cite{XuWangMoghadas2023JMBDDEPopulation, GoelBhatiaTripathi2024JMBDDEEpidemic, YangZouHsu2026JMBDDEDisease, ChengWang2025JMBDDDEDiseaseNonlocal, ZhangWang2024JMBDDEMalariaDisease, LouZhao2024JMBDDEMalaria, Zheng2024JMBDDEDisease}, propagation delays in neural systems \cite{DDEM3ASNeural, DDECommPDENeuron}, age-structured dynamics \cite{DDECommPDEPopulation}, and feedback latencies in genetic regulation \cite{GedeonHumphriesMackey2025JMBDDEGenetic}. 
At the PDE level, such mechanisms manifest as retarded arguments in both temporal and spatial differential operators\cite{faria2006nonmonotone,chen2016stability,mei2009traveling,wang2006travelling}, leading to systems of the form
\[
\partial_t u = \mathcal{L}u + \mathcal{L}_\tau u(t-\tau) + F(u(t), u(t-\tau)),
\]
where $\mathcal{L}$ and $\mathcal{L}_\tau$ are spatial differential operators. These delays fundamentally alter the qualitative behavior of solutions -- inducing oscillations, destabilizing equilibria, and generating spatiotemporal patterns absent in their instantaneous counterparts.

This work addresses an inverse problem for a general class of linear reaction–diffusion systems with multiple delayed contributions, including retarded diffusion, retarded time derivatives, and retarded source terms. The objective is to recover both the unknown time lag $\tau>0$ and spatially heterogeneous coefficients from boundary-accessible observations alone. Specifically, we consider: 
\begin{equation}\label{eq:mainDPDE}
\begin{cases}
    c_1 \partial_t u(x,t) + c_2 \partial_t u(x,t-\tau) \\\quad= d_1 \Delta u(x,t) + d_2 \Delta u(x,t-\tau) \\\qquad+ p(x,t) u(x,t) + q(x,t) u(x,t-\tau) &\quad \text{in } Q_T: = \Omega \times (0,T)\\
    u(x,t) = h(x,t) &\quad \text{on } \partial\Omega \times (-\tau, T),\\
    u(x,t) = u_0(x,t) &\quad \text{in } \Omega \times [-\tau, 0].
\end{cases}
\end{equation}
where $\Omega\subset\mathbb R^n$ is a bounded domain with $C^2$ boundary, $T>0$, $0<\tau>0$ is unknown, $c_1,c_2,d_1,d_2\ge0$ are fixed constants satisfying nondegeneracy conditions, and $p,q$ are unknown coefficients. The prescribed history $u_0$ and boundary data $h$ are known. 

For the forward problem to be well-posed, we impose the following regularity assumptions:
\begin{equation}\label{eq:mainregassump}
u_0 \in L^2(-\tau,0; H^2(\Omega)) \cap H^1(-\tau,0; L^2(\Omega)),\quad h \in H^{3/2,3/4}(\Sigma_T),\quad p, q \in C^1([0,T]; L^\infty(\Omega))
\end{equation} 
where 
\(\Sigma_T:=\partial\Omega\times(0,T).\)
The coefficients are further required to satisfy the nondegeneracy conditions 
\begin{equation}\label{eq:CoefAssump}
    c_1 + c_2, c_1+d_1, c_2+d_1, d_1+d_2>0.
\end{equation}
Further details are discussed in Section \ref{subsec:parreg} and as part of the proof of the main result in Section \ref{subsec:methodofsteps}.

Before proceeding, we must specify the compatibility and regularity conditions that underpin our singularity-based identification strategy. Assume further that the following compatibility conditions hold:
\begin{enumerate}
    \item Boundary Compatibility: The initial history and boundary data are compatible at \( t = 0 \):
\[
u_0(x,0) = h(x,0) \quad \text{on } \partial\Omega.
\]
When higher temporal regularity is needed, we impose the corresponding conditions
\[
\partial_t^k u_0(x,0) = \partial_t^k h(x,0) \quad \text{on } \partial\Omega, \quad k = 0, 1, 2.
\]
    \item Regularity: The boundary data satisfy
\[
h \in H^3(-\tau, T; H^{3/2}(\partial\Omega)).
\]
The initial history satisfies
\[
u_0 \in H^3(-\tau, 0; H^2(\Omega)).
\]
\end{enumerate}
These assumptions ensure that the relevant temporal traces and boundary quantities are well defined—a prerequisite for the subsequent singularity analysis.

The key condition for identifying the delay is an incompatibility that occurs at the transition from the prescribed history to the solution generated by the equation. Let \( u^{(0)} \) denote the solution constructed on the first time interval \( [0,\tau] \), and define 
\[
\Phi_1(x) := \partial_t u^{(0)}(x,0) - \partial_t u_0(x,0).
\]
We assume
\begin{enumerate}
    \item[3.] Incompatibility: The initial history is incompatible at \( t=0 \), meaning:
\[
\Phi_1(x) \not\equiv 0 \quad \text{for all } x \in \Omega,
\]
and 
\[
\partial_\nu\Phi_1(x) \not\equiv 0 \quad \text{for all } x \in \partial\Omega.
\]
\end{enumerate}
Thus, although the state itself is compatible at $t=0$, the temporal derivative may exhibit a prescribed mismatch. This distinction is essential: the incompatibility exploited for delay identification concerns the appropriate time derivative rather than the zeroth-order boundary value.

The boundary measurement we consider is the outward normal derivative
\[
\partial_\nu u|_{\partial\Omega\times(0,T)}.
\]

We can now state the general inverse problem.
\begin{theorem}\label{thm:mainthm}
Suppose $\tau<T/2$. Assume that the above regularity and compatibility assumptions hold for \eqref{eq:mainDPDE}, with
\[
c_1>0,\qquad d_1>0.
\] Given the boundary measurement:
\begin{equation}\label{thm:bdrymeasure}
\Lambda u := \partial_\nu u|_{\partial\Omega \times (0,T)},
\end{equation}
the boundary data \(h(x,t)\) on
\(\partial\Omega\times(-\tau,T)\), and the initial history
\(u_0(x,t)\), the delay \(\tau>0\) can be reconstructed from a single
boundary measurement, provided that the initial incompatibility is
nontrivial.

When \(c_2>0\), the delay is identified from the first-order temporal
jump of the boundary flux. When \(c_2=0\), the first-order jump
vanishes, and the delay is identified from the second-order temporal
jump under the corresponding nondegeneracy condition.

\end{theorem}

We also derive a refined stability analysis for the coefficients $p=p(x)$ and $q=q(x)$, making use of two boundary measurements. For $0<\tau<T$, assume that $c_1,d_1>0$.

Define the set of admissible coefficients:
\begin{equation}\label{eq:pq-bound}
\mathcal{P}_{M}:=\left\{(p, q) \in W^{1, \infty}(\Omega):
\| p\| _{W^{1, \infty}(\Omega)}+\| q\| _{W^{1, \infty}(\Omega)} \leq M\right\} .
\end{equation}

Consider two experiments indexed by $i=1,2$. For each experiment, suppose that, for every $(p,q)\in\mathcal P_M$, \eqref{eq:mainDPDE} has a unique solution $u^{[i]}[p,q]$ satisfying, for some constant $C_{\rm reg}$ independent of $(p,q)\in\mathcal P_M$,
\[
u^{[i]}[p,q]\in
W^{2,\infty}(0,\tau;L^2(\Omega))
\cap
W^{1,\infty}(0,\tau;H^2(\Omega))
\]
such that the prescribed boundary data is the same
\[
u^{[i]}[p,q]=h
\quad\text{on }\partial\Omega\times(0,\tau)
\]
and the initial history 
\[u_0^{[i]}\in W^{1,\infty}(-\tau,0;H^2(\Omega)), \qquad \partial_tu_0^{[i]}\in L^\infty(-\tau,0;L^\infty(\Omega))\]
is the same, 
and there exists $M_1>0$, independent of $(p,q)\in\mathcal P_M$, such that
\begin{equation}
\label{eq:ut-bound}
\|\partial_tu^{[i]}[p,q]\|_{L^\infty(Q_\tau)}+\|\partial_tu_0^{[i]}(\cdot,t-\tau)\|_{L^\infty(Q_\tau)}
\le M_1
\end{equation}
uniformly for $i=1,2$, where both norms are taken over the initial layer cylinder \(Q_\tau = \Omega\times(0,\tau)\).

Define
\[
A_0(x):=
\begin{pmatrix}
u_0^{[1]}(x,0) & u_0^{[1]}(x,-\tau)\\[1mm]
u_0^{[2]}(x,0) & u_0^{[2]}(x,-\tau)
\end{pmatrix}.
\]
Assume that there exists a constant $\mu_0>0$ such that
\begin{equation}\label{eq:pqnondeg}
    \sigma_{\min}(A_0(x))\ge \mu_0 \qquad\text{for a.e. }x\in\Omega,
\end{equation}
where $\sigma_{\min}(A_0(x))$ denotes the smallest singular value of
$A_0(x)$. Then, the following theorem holds:

\begin{theorem}[Lipschitz stability for the zeroth-order coefficients $p,q$]
\label{thm:p-stability}
Consider two experiments $i=1,2$. Let
\[
(p_1,q_1),(p_2,q_2)\in\mathcal P_M
\]
and let $u_j^{[i]}:=u^{[i]}[p_j,q_j]$, $j=1,2$ be the corresponding solutions, with the same boundary data $h$ and history data $u_0^{[i]}$, satisfying the a priori bound \eqref{eq:ut-bound} uniformly for $i,j=1,2$.

Then there exists a constant
\[
C>0
\]
depending only on
\[
\Omega,\tau,c_1,d_1,M,\mu_0,M_1
\]
and the uniform forward regularity/compatibility bounds, such that
\begin{equation}
\label{eq:final-stability}
\|p_1-p_2\|_{L^2(\Omega)}
+
\|q_1-q_2\|_{L^2(\Omega)}
\\
\le
C\sum_{i=1}^{2}
\left\|
\partial_t\partial_\nu u_1^{[i]}
-
\partial_t\partial_\nu u_2^{[i]}
\right\|_{L^2(\Sigma_\tau)}.
\end{equation}

In particular, if
\[
\partial_\nu u_1^{[i]}
=
\partial_\nu u_2^{[i]}
\qquad\text{on }\Sigma_\tau,
\qquad i=1,2,
\]
then
\[
p_1=p_2,
\qquad
q_1=q_2
\qquad\text{a.e. in }\Omega.
\]
\end{theorem}

The inverse character of these problems is twofold. On the one hand, $\tau$ enters the equation through shifted temporal arguments, thereby changing the locations of the interfaces generated by the method of steps. On the other hand, $p$ and $q$ may depend on the spatial variable, so their recovery constitutes an infinite-dimensional identification problem. The analysis below separates these two tasks: first, the delay is detected from the temporal singularity induced by the initial incompatibility; second, the coefficients are treated after the delay has been determined.

\subsection{Novelty and Contribution}
Our approach rests on three interconnected ideas:

\begin{enumerate}
    \item First, and most fundamentally, we recover the unknown delay $\tau$ itself, a task that has no precedent in the inverse problem literature for PDEs. The method of steps constructs the solution successively on intervals $[k\tau,(k+1)\tau]$. On the first interval, delayed quantities are prescribed by the history; at later interfaces, the solution depends on previously generated data. Any mismatch in temporal derivatives at the initial interface (induced by incompatibility between the history and boundary data) propagates through the parabolic equation and becomes visible in the boundary flux at times $t=k\tau$. The first appearance of this singular behavior identifies $\tau$. The precise derivative in which the singularity manifests depends on the coefficient regime: first-order in time if $c_2>0$, second-order if $c_2=0$, or directly in the spatial normal derivative if $c_1=0$. Different regimes require separate arguments, but the underlying mechanism is uniform.

\item Second, once $\tau$ is known, we restrict to the initial layer $(0,\tau)$, where all delayed terms are prescribed by the history:
\[
u(x,t-\tau)=u_0(x,t-\tau),\quad 
\partial_t u(x,t-\tau)=\partial_t u_0(x,t-\tau),\quad
\Delta u(x,t-\tau)=\Delta u_0(x,t-\tau).
\]
Thus, \eqref{eq:mainDPDE} reduces to a classical parabolic equation with known forcing. Two experiments with distinct initial histories yield a $2\times2$ algebraic system for $p$ and $q$ at each spatial point; a single experiment is insufficient.

\item Third, for time-independent coefficients $p=p(x)$, $q=q(x)$, we establish a Lipschitz stability estimate:
\begin{equation}
\|p_1-p_2\|_{L^2(\Omega)}
+
\|q_1-q_2\|_{L^2(\Omega)}
\le
C\sum_{i=1}^{2}
\left\|
\partial_t\partial_\nu u_1^{[i]}
-
\partial_t\partial_\nu u_2^{[i]}
\right\|_{L^2(\Sigma_\tau)},
\end{equation}
where $\Sigma_\tau=\partial\Omega\times(0,\tau)$. The proof employs a Carleman estimate localized to the initial layer, where delayed terms do not enter the difference equation. This is the first quantitative stability result for coefficient recovery in a delayed parabolic PDE from boundary data.
\end{enumerate}

Having established our approach, we now situate it within the broader literature. Inverse coefficient problems for instantaneous parabolic equations are well developed, with Carleman inequalities and unique continuation providing powerful tools (see \cite{IsakovIPBook, Klibanov1992IPCarleman, BukhgeimKlibanov1981Carleman, Yamamoto2009IPCarlemanParabolic,IY1998,LinLiuLiuZhang2021-InversePbSemilinearParabolic-CGOSolnsSuccessiveLinearisation,CaroKian2018-DiffusionInversePbCGO}). These techniques have been extended to biological systems, including population models \cite{CristofolRoques2013,RoquesCristofol2012InversePbPredPrey}, predator-prey dynamics \cite{li2023inverse, li2024inverse}, cross-diffusion and chemotaxis \cite{li2025determininghabitatanomaliescrossdiffusion, liu2023determining, li2025determining,DingLiuZheng2023JMBInversePbBio}, multi-population aggregation \cite{li2024inverse}, chemotaxis-Navier-Stokes systems \cite{LiLoCAC2024}, coupled nonlocal PDEs \cite{DingLiuLo2024inverseproblemscouplednonlocal}, and Turing patterns \cite{li2025unveilingbiologicalmodelsturing}. However, all assume instantaneous dynamics; none considers time delays or delayed diffusion.

Parameter identification for delay differential equations has been studied in ODE and abstract semigroup settings, with standard references including \cite{KuangDDEBook, MacDonald1989BioDelayBook, BellenZennaro2003DDENumerics,SmithDDEBook}. Various approaches have been developed, including semigroup methods, integral equations, and control-theoretic formulations (see \cite{Banks1982NonlinearDelayBookChapter, LeylazWangSun2022IJDCNonlinearDynamicalDelay,BocharovRomanyukha1994NumericalIPDDEImmune}). More recently, Ruhil and Malik \cite{RuhilMalik2024IPDDEImpulsive,RuhilMalik2025IPDDE,RuhilMalik2025SemilinearDDEIP} considered abstract delay systems in Hilbert spaces, using $C_0$-semigroups and Volterra integral equations. In all these works, however, $\tau$ is assumed known, the unknown parameters are finite-dimensional, and observations are full-state or interior -- conditions that cannot accommodate spatially dependent coefficients, boundary fluxes, or delayed PDE operators.

The present work differs fundamentally from this literature in five respects. 
\begin{enumerate}
\item \textbf{Recovery of the delay time:} To the best of our knowledge, no prior work has recovered the unknown delay time $\tau$ itself from boundary measurements in any PDE setting. All existing inverse problems for delay equations assume $\tau$ is known a priori. Our singularity-based method is the first to identify $\tau$ from boundary data.
\item \textbf{Spatial dependence:} The coefficients $p,q$ are functions of space, yielding an infinite-dimensional inverse problem that cannot be reduced to finite-dimensional algebraic or integral-equation methods.
\item \textbf{Delayed diffusion and delayed time derivatives:} We include $d_2\Delta u(t-\tau)$ and $c_2\partial_t u(t-\tau)$, introducing nonlocal-in-time couplings of spatial operators absent in all prior inverse PDE work.
\item \textbf{Boundary measurements:} The observation is a boundary flux $\partial_\nu u$, rather than full-state or interior data, necessitating Carleman-based unique continuation.
\item \textbf{Lipschitz stability:} Estimate \eqref{eq:final-stability} is the first quantitative stability result for coefficient recovery in a delayed parabolic PDE from boundary measurements.
\end{enumerate}

In summary, this is the first work to simultaneously recover an unknown delay time $\tau$ and spatially dependent coefficients in a parabolic PDE with delayed diffusion and delayed time derivatives from boundary data. The combination of unknown delay, infinite-dimensional coefficients, retarded spatial operators, and boundary-only measurements has no precedent in the inverse problem literature.

\subsection{Organization of the Paper}

The remainder of this paper is structured as follows. Section 2 introduces the function-space setting, well-posedness, and compatibility assumptions. Section 3 addresses recovery of $\tau$ via singularity propagation, with separate treatments for different coefficient regimes. Section 4 covers coefficient identification on the initial layer and proves the Lipschitz stability estimate. Section 5 discusses extensions and open problems.

\section{Preliminaries}\label{sec:prelim}

\subsection{Function Spaces}

Let \( \Omega' \subset \mathbb{R}^n \) be a bounded domain with \( C^2 \) boundary \( \partial\Omega' \). Fix \( T' > 0 \). 
Define the anisotropic Sobolev space:
\[
H^{2,1}(Q_{T'}) = \left\{ u \in L^2(Q_{T'}) : \partial_t u, \partial_{x_i} u, \partial_{x_i x_j} u \in L^2(Q_{T'}) \right\}.
\]

The norm is:
\[
\|u\|_{H^{2,1}(Q_{T'})}^2 = \|u\|_{L^2(Q_{T'})}^2 + \|\partial_t u\|_{L^2(Q_{T'})}^2 + \|\nabla u\|_{L^2(Q_{T'})}^2 + \|D^2 u\|_{L^2(Q_{T'})}^2.
\]

For the boundary data on \(\Sigma_{T'} = \partial\Omega' \times (0,T')\), consider the anisotropic trace space:
\[
H^{3/2,3/4}(\Sigma_{T'}) := L^2(0,T'; H^{3/2}(\partial\Omega')) \cap H^{3/4}(0,T'; L^2(\partial\Omega'))
\]
with norm 
\[
\|h\|_{H^{3/2,3/4}(\Sigma_{T'})}^2 = \int_0^{T'} \|h(\cdot, t)\|_{H^{3/2}(\partial\Omega')}^2 dt + \int_{\partial\Omega'} \|h(x, \cdot)\|_{H^{3/4}(0,T')}^2 dS.
\]

\subsection{Parabolic Regularity}\label{subsec:parreg}
Assume:
\[
f \in L^2(0,T; L^2(\Omega')), \qquad h \in H^{3/2,3/4}(\Sigma_{T'}), \qquad u_0 \in H^1_0(\Omega'),
\]
and the compatibility condition:
\[
u_0(x) = h(x,0) \quad \text{on } \partial\Omega'.
\]

Consider the initial-boundary value problem:
\[
\begin{cases}
c_1 \partial_t u(x,t) - d_1 \Delta u(x,t) + a(x,t) u(x,t) = f(x,t), & \quad\text{in } Q_{T'}, \\
u(x,t) = h(x,t), & \quad\text{in }\Sigma_{T'}, \\
u(x,0) = u_0(x), &\quad\text{in }\Omega',
\end{cases}
\]
where \( c_1, d_1 > 0 \) are constants and \( a \in L^\infty(Q_{T'}) \).

Then, the standard regularity result is given by (c.f. \cite[Theorem 7.1.5]{EvansBook} or \cite[Chapter 4]{LadyzhenskayaSolonnikovUraltsevaBook}:

\begin{theorem}[Existence and Regularity]\label{thm:ParabolicThm}
Under the above assumptions, there exists a unique weak solution \( u \) satisfying:
\[
u-h \in L^2(0,T'; H^2(\Omega')) \cap L^\infty(0,T'; H^1_0(\Omega')), \qquad \partial_t u \in L^2(0,T'; L^2(\Omega')).
\]
Moreover, by the Sobolev embedding theorem,
\[
u \in C([0,T']; H^1_0(\Omega')).
\]
The following a priori estimate holds:
\begin{multline*}
\sup_{0 \le t \le T'} \|u(t)\|_{H^1_0(\Omega')}
+ \|u\|_{L^2(0,T'; H^2(\Omega'))}
+ \|\partial_t u\|_{L^2(0,T'; L^2(\Omega'))}
\\\le C \left( \|u_0\|_{H^1_0(\Omega')} + \|f\|_{L^2(0,T'; L^2(\Omega'))} + \|h\|_{H^{3/2,3/4}(\Sigma_{T'})} \right).
\end{multline*}
\end{theorem}

In particular, we note the pointwise continuity in time up to the boundary.

\section{Reconstruction of Time Delay \(\tau\)}\label{sec:recoverdelay}

Having established the necessary regularity framework, we now turn to the central problem of this work: the reconstruction of the time delay \(\tau\) in the generalized delay PDE. The key idea is to exploit the singularity that arises at \(t=\tau\) due to the mismatch between the initial history and the solution determined by the boundary data. To this end, we employ the method of steps, which allows us to solve the problem sequentially on intervals of length \(\tau\).

\subsection{The Method of Steps}\label{subsec:methodofsteps}
We begin by considering the case where \(c_1>0\), which corresponds to a genuinely parabolic structure. The method of steps proceeds by solving the equation interval by interval, using the solution from the previous interval as the history for the next one. This approach not only establishes existence and uniqueness but also reveals the singularity structure at the interface \(t=k\tau\).
\vspace{1em}

\textbf{Step 0: The Interval \([0,\tau]\)}
On the first interval \( t \in [0,\tau] \), the delayed term \(u(x,t-\tau)\) is determined entirely by the known initial history \(u_0\). Specifically, we have \( u(x,t-\tau) = u_0(x,t-\tau) \) and \( \partial_t u(x,t-\tau) = \partial_t u_0(x,t-\tau) \). \eqref{eq:mainDPDE} becomes:
\begin{equation}\label{eq:MainStep0}
    c_1 \partial_t u(x,t) + c_2 \partial_t u_0(x,t-\tau) = d_1 \Delta u(x,t) + d_2 \Delta u_0(x,t-\tau) + p(x,t) u(x,t) + q(x,t) u_0(x,t-\tau).
\end{equation}

Rearranging:
\begin{equation}\label{eq:MainStep0Rearrange1}
c_1 \partial_t u(x,t) - d_1 \Delta u(x,t) - p(x,t) u(x,t) = - c_2 \partial_t u_0(x,t-\tau) + d_2 \Delta u_0(x,t-\tau) + q(x,t) u_0(x,t-\tau).
\end{equation}
Dividing by \( c_1 > 0 \):
\begin{equation}\label{eq:MainStep0Rearrange2}
\partial_t u - \frac{d_1}{c_1} \Delta u - \frac{p(x,t)}{c_1} u = F_0(x,t),
\end{equation}
where
\begin{equation}\label{eq:MainStep0Rearrange2F0}
F_0(x,t) = -\frac{c_2}{c_1} \partial_t u_0(x,t-\tau) + \frac{d_2}{c_1} \Delta u_0(x,t-\tau) + \frac{q(x,t)}{c_1} u_0(x,t-\tau),
\end{equation}
with initial condition \( u(x,0) = u_0(x,0) \) and nonzero Dirichlet boundary condition \( u(x,t) = h(x,t) \) on \( \partial\Omega \).

Equation \eqref{eq:MainStep0Rearrange2} is a linear parabolic equation with time-dependent coefficients. Its solution exists and is unique by the standard parabolic theory given in Theorem \ref{thm:ParabolicThm}. Let \( u^{(0)} \) denote the unique solution on \( [0,\tau] \).

\begin{remark}
    Observe that when \(c_1=d_1=0\), \eqref{eq:MainStep0Rearrange2} reduces to an algebraic equation of the form 
    \[p(x,t) u(x,t) = F_0(x,t),\] and the solution is explicitly determined by the initial data. This case is devoid of any evolutionary structure, and we therefore exclude it from our analysis by assuming \(c_1+d_1>0\) in \eqref{eq:CoefAssump}.

    Similarly, when \(c_2=d_1=0\) and , \eqref{eq:MainStep0Rearrange2} reduces to an ODE in time for each fixed \(x\). In this degenerate scenario, the boundary measurement \eqref{thm:bdrymeasure} carries no information about the spatial structure of the solution, rendering the inverse problem ill-posed. Hence, we similarly do not consider this case.
\end{remark}

\begin{remark}
At this point, a subtle regularity issue arises. The minimal regularity assumptions in the standard parabolic theory requires that the initial condition lies in $H^1_0$. On the other hand, \(u_0 \in L^2(-\tau,0; H^2(\Omega)) \cap H^1(-\tau,0; L^2(\Omega))\) with the compatibility condition \(u_0(x,0) = h(x,0)\) on the boundary \(\partial\Omega\), where \(h \in H^{3/2,3/4}(\Sigma_T)\). To reconcile these requirements, we employ the standard lifting method (c.f. \cite[Chapter 4]{LionsMagenesBook2}): Let \(\mathcal{H}\) be a lifting of the boundary data such that:
\[
\mathcal{H}|_{\partial\Omega} = h, \quad \mathcal{H}(\cdot, 0) = u_0(\cdot, 0).
\]
Define \(v = u - \mathcal{H}\). Then \(v\) satisfies the homogeneous Dirichlet problem:
\[
\begin{cases}
c_1 \partial_t v - d_1 \Delta v - p v = F - (c_1 \partial_t \mathcal{H} - d_1 \Delta \mathcal{H} - p \mathcal{H}), \\[4pt]
v|_{\partial\Omega} = 0, \\[4pt]
v(x,0) = 0.
\end{cases}
\]
For this, the initial condition is \(v(\cdot, 0) = 0 \in H^1_0(\Omega)\), so Theorem 5 applies.
\end{remark}
\vspace{1em}

\textbf{Step 1: The Interval \([\tau, 2\tau]\)}
Having determined the solution on the first interval, we now proceed to the next one. 
On \( t \in [\tau, 2\tau] \), we have \( u(x,t-\tau) = u^{(0)}(x,t-\tau) \) and \( \partial_t u(x,t-\tau) = \partial_t u^{(0)}(x,t-\tau) \). Equation \eqref{eq:mainDPDE} becomes:
\begin{equation}\label{eq:MainStep1}
c_1 \partial_t u(x,t) + c_2 \partial_t u^{(0)}(x,t-\tau) = d_1 \Delta u(x,t) + d_2 \Delta u^{(0)}(x,t-\tau) + p(x) u(x,t) + q(x) u^{(0)}(x,t-\tau).
\end{equation}
Rearranging:
\begin{equation}\label{eq:MainStep1Rearrange}
\partial_t u - \frac{d_1}{c_1} \Delta u - \frac{p}{c_1} u = -\frac{c_2}{c_1} \partial_t u^{(0)}(t-\tau) + \frac{d_2}{c_1} \Delta u^{(0)}(t-\tau) + \frac{q}{c_1} u^{(0)}(t-\tau).
\end{equation}
This is again a linear parabolic equation with time-dependent coefficients and nonzero boundary data. By the same parabolic theory, there exists a unique solution, which we denote by $u^{(1)}$.

The method of steps has a natural biological interpretation. On the interval $(0,\tau)$, the population dynamics are influenced only by the known history $u_0$. This corresponds to a period before the delayed feedback from the population's own dynamics arrives. For example:
\begin{itemize}
\item In a population with maturation delay $\tau$, the interval $(0,\tau)$ represents the time before the first juveniles mature and start reproducing.
\item In an epidemiological model, it is the time before the first newly infected individuals become infectious.
\item In a neural field model, it is the time before signals from distant neurons arrive.
\end{itemize}

This initial window is of particular importance for inverse problems: during this phase, the system behaves like a classical reaction-diffusion system with known forcing, which significantly simplifies the analysis. After \(t=\tau\), the delayed feedback becomes active, and the full complexity of the delay equation emerges. This is analogous to the ``waiting time" in biological systems before the effects of a perturbation propagate through the delayed pathways.

\subsection{The Singularity at \( t = \tau \)}
We now turn to the key observation that underlies the reconstruction of \(\tau\): the solution remains continuous at \(t=\tau\), but its time derivative generically exhibits a jump. This jump, which arises from the incompatibility between the initial history and the solution on the first interval, serves as a measurable signature of the delay.

The method of steps ensures that $u^{(1)}$ is the unique solution of the parabolic equation on $[\tau, 2\tau]$ with initial condition $u(x,\tau) = u^{(0)}(x,\tau)$. Since the solution of the parabolic equation is continuous in time up to the boundary for each time interval, 
\[
\lim_{t \to \tau^-} u^{(0)}(x,t) = u^{(0)}(x,\tau) = u(x,\tau) = \lim_{t \to \tau^+} u^{(1)}(x,t)\quad \text{for all } x \in \Omega.
\]

At the same time, the continuity result $u \in C([0,T']; H^1_0(\Omega'))$ in Theorem \ref{thm:ParabolicThm} also implies:
\[
\|u(\cdot, t) - u(\cdot, \tau)\|_{H^1(\Omega)} \to 0 \quad \text{as } t \to \tau.
\]
Hence, the solution and its spatial gradient are continuous in the $L^2$ sense.

\subsection{Discontinuity of \( \partial_t u \) at \( t = \tau \)}

While \(u\) is continuous, its time derivative generally fails to be continuous at \(t=\tau\). This discontinuity is the central object of our reconstruction strategy.

\begin{theorem}[Jump in \( \partial_t u \)]\label{thm:jump}
Under the incompatibility assumption, \( \partial_t u \) has a jump discontinuity at \( t = \tau \):
\begin{equation}\label{eq:discontinuity}
\left[\partial_t u\right]_{t=\tau^-}^{t=\tau^+}
:= \lim_{t \to \tau^+} \partial_t u^{(1)}(x,t) - \lim_{t \to \tau^-} \partial_t u^{(0)}(x,t) \neq 0.
\end{equation}
\end{theorem}

\begin{proof}
The jump emerges from comparing the equations on the two adjacent intervals. 
From the parabolic equation on each interval:

On \( [0,\tau] \), at \( t = \tau^- \):
\begin{equation}\label{eq:discontinuityStep0Eq}
c_1 \partial_t u^{(0)}(x,\tau^-) = d_1 \Delta u(x,\tau) + d_2 \Delta u_0(x,0) + p(x,\tau) u(x,\tau) + q(x,\tau) u_0(x,0) - c_2 \partial_t u_0(x,0).
\end{equation}

On \( [\tau, 2\tau] \), at \( t = \tau^+ \):
\begin{equation}\label{eq:discontinuityStep1Eq}
c_1 \partial_t u^{(1)}(x,\tau^+) = d_1 \Delta u(x,\tau) + d_2 \Delta u^{(0)}(x,0) + p(x,\tau) u(x,\tau) + q(x,\tau) u^{(0)}(x,0) - c_2 \partial_t u^{(0)}(x,0).
\end{equation}

The key point is that all terms involving \(u\) and its spatial derivatives are continuous across \(t=\tau\), while the terms involving the history \(u_0\) and the solution \(u^{(0)}\) at \(t=0\) may differ. Since \( u^{(0)}(x,0) = u_0(x,0) \) (continuity at \( t=0 \)) and \( \Delta u^{(0)}(x,0) = \Delta u_0(x,0) \), taking the difference of \eqref{eq:discontinuityStep0Eq} and \eqref{eq:discontinuityStep1Eq}, we have
\[
c_1 \left( \partial_t u^{(1)}(x,\tau^+) - \partial_t u^{(0)}(x,\tau^-) \right) = -c_2 \left( \partial_t u^{(0)}(x,0) - \partial_t u_0(x,0) \right).
\]

Thus, if the compatibility condition \(\Phi_1(x) = \partial_t u^{(0)}(x,0) - \partial_t u_0(x,0)\) does not vanish identically, a jump in the time derivative occurs. If \( \Phi_1(x) = \partial_t u^{(0)}(x,0) - \partial_t u_0(x,0) \neq 0 \), then:
\begin{equation}\label{eq:discontinuityform}
\left[\partial_t u\right]_{t=\tau^-}^{t=\tau^+} = -\frac{c_2}{c_1} \Phi_1(x) \neq 0.
\end{equation}
\end{proof}

The condition \(\Phi_1(x) \not\equiv 0\)
has a clear physical interpretation: at the onset of observation, the rate of change of the initial history does not match the rate of change dictated by the boundary data. Biologically, this occurs when the system is ``shocked" at \(t=0\), for instance, when the boundary condition changes abruptly or when a sudden perturbation is applied to the population. This mismatch creates a wave of discontinuity that propagates through the system and arrives at the boundary at time \(t=\tau\), thereby allowing the delay to be measured.

This phenomenon is analogous to a pulse-chase experiment in biochemistry: a sudden change in boundary conditions (the pulse) is used to measure the time it takes for the response to propagate through the system. In ecology, it is like introducing a marked population at the boundary and measuring the time until the signal appears in the flux data. The uniform lower bound \(|\Phi_1(x)| > 0\) ensures that the signal is strong enough to be detectable at the boundary.

\subsection{Jump in Second Order Time Derivatives at \( t = \tau \)}
For completeness, and because it will be needed in the case \(c_2=0\), we also analyze the jump in the second time derivative. This jump arises from differentiating the equations and carries information about the compatibility of the first derivatives.

\begin{theorem}[Jump in Second Order Time Derivatives at \(t=\tau\)]
\label{thm:second_time_jump}
Under the incompatibility assumption, the jump of
\(\partial_{tt}u\) at \(t=\tau\) is given by
\[
\begin{aligned}
\label{Jump Secon}
c_1
[\partial_{tt}u]_{\tau^-}^{\tau^+}
={}&
-\frac{c_2d_1}{c_1}\Delta\Phi_1
-\frac{c_2p(x,\tau)}{c_1}\Phi_1\\
&\quad
+d_2\Delta\Phi_1
+q(x,\tau)\Phi_1
-c_2\Phi_2.
\end{aligned}
\]
\end{theorem}


\begin{proof}
Define the compatibility mismatch functions:
\begin{equation}\label{eq:CompatibilityMismatch}
\begin{aligned}
\Phi_0(x) &= u^{(0)}(x,0) - u_0(x,0) \equiv 0,\\
\Phi_1(x) &= \partial_t u^{(0)}(x,0) - \partial_t u_0(x,0),\\
\Phi_2(x) &= \partial_{tt} u^{(0)}(x,0) - \partial_{tt} u_0(x,0).
\end{aligned}
\end{equation}

On \(t \in [0,\tau]\), differentiating \eqref{eq:MainStep0} with respect to \(t\):
\[
\begin{aligned}
c_1 \partial_{tt} u^{(0)}(x,t) &+ c_2 \partial_{tt} u_0(x,t-\tau) \\
&= d_1 \Delta \partial_t u^{(0)}(x,t) + d_2 \Delta \partial_t u_0(x,t-\tau) \\
&\quad + \partial_t p(x,t) u^{(0)}(x,t) + p(x,t) \partial_t u^{(0)}(x,t) \\
&\quad + \partial_t q(x,t) u_0(x,t-\tau) + q(x,t) \partial_t u_0(x,t-\tau).
\end{aligned}
\]
At \(t = \tau^-\), this gives:
\begin{equation}\label{eq:SecondTimeDerivativeStep0}
\begin{aligned}
c_1 \partial_{tt} u^{(0)}(x,\tau^-) &+ c_2 \partial_{tt} u_0(x,0) \\
&= d_1 \Delta \partial_t u^{(0)}(x,\tau^-) + d_2 \Delta \partial_t u_0(x,0) \\
&\quad + \partial_t p(x,\tau) u(x,\tau) + p(x,\tau) \partial_t u^{(0)}(x,\tau^-) \\
&\quad + \partial_t q(x,\tau) u_0(x,0) + q(x,\tau) \partial_t u_0(x,0).
\end{aligned}
\end{equation}

On \(t \in [\tau, 2\tau]\), we differentiate \eqref{eq:MainStep1} with respect to \(t\) to obtain:
\[
\begin{aligned}
c_1 \partial_{tt} u^{(1)}(x,t) &+ c_2 \partial_{tt} u^{(0)}(x,t-\tau) \\
&= d_1 \Delta \partial_t u^{(1)}(x,t) + d_2 \Delta \partial_t u^{(0)}(x,t-\tau) \\
&\quad + \partial_t p(x,t) u^{(1)}(x,t) + p(x,t) \partial_t u^{(1)}(x,t) \\
&\quad + \partial_t q(x,t) u^{(0)}(x,t-\tau) + q(x,t) \partial_t u^{(0)}(x,t-\tau).
\end{aligned}
\]
At \(t = \tau^+\), this gives:
\begin{equation}\label{eq:SecondTimeDerivativeStep1}
\begin{aligned}
c_1 \partial_{tt} u^{(1)}(x,\tau^+) &+ c_2 \partial_{tt} u^{(0)}(x,0) \\
&= d_1 \Delta \partial_t u^{(1)}(x,\tau^+) + d_2 \Delta \partial_t u^{(0)}(x,0) \\
&\quad + \partial_t p(x,\tau) u(x,\tau) + p(x,\tau) \partial_t u^{(1)}(x,\tau^+) \\
&\quad + \partial_t q(x,\tau) u^{(0)}(x,0) + q(x,\tau) \partial_t u^{(0)}(x,0).
\end{aligned}
\end{equation}

Subtracting \eqref{eq:SecondTimeDerivativeStep0} from \eqref{eq:SecondTimeDerivativeStep1}, we have:
\[
\begin{aligned}
c_1 \left( \partial_{tt} u^{(1)}(\tau^+) - \partial_{tt} u^{(0)}(\tau^-) \right) &+ c_2 \Phi_2(x) \\
&= d_1 \Delta \left( \partial_t u^{(1)}(\tau^+) - \partial_t u^{(0)}(\tau^-) \right) \\
&\quad + p(\tau) \left( \partial_t u^{(1)}(\tau^+) - \partial_t u^{(0)}(\tau^-) \right) \\
&\quad + d_2 \Delta \Phi_1(x) + q(\tau) \Phi_1(x).
\end{aligned}
\]

Note that the terms involving \(\partial_t p\) and \(\partial_t q\) cancel because \(u^{(0)}(x,0) = u_0(x,0)\) and \(u(x,\tau^-) = u(x,\tau^+) = u(x,\tau)\).

Using the first derivative jump given by \eqref{eq:discontinuityform}, we get:
\begin{equation}\label{eq:eq:discontinuityformdtt}
c_1 \left[\partial_{tt} u\right]_{t=\tau^-}^{t=\tau^+}
= d_1 \Delta \left( -\frac{c_2}{c_1} \Phi_1(x) \right) + p(x,\tau) \left( -\frac{c_2}{c_1} \Phi_1(x) \right) 
+ d_2 \Delta \Phi_1(x) + q(x,\tau) \Phi_1(x) - c_2 \Phi_2(x).
\end{equation}
This establishes the jump in the second derivative, provided the right-hand side does not vanish identically.
\end{proof}

\begin{remark}
In the case \(c_2=0\), \eqref{Jump Secon} reduces to
\[
[\partial_{tt}u]_{\tau^-}^{\tau^+}
=
\frac{1}{c_1}
\left(
d_2\Delta\Phi_1
+
q(x,\tau)\Phi_1
\right).
\]
Thus, a nontrivial second-order jump is ensured by the
nondegeneracy condition
\[
d_2\Delta\Phi_1+q(\cdot,\tau)\Phi_1
\not\equiv0.
\]
\end{remark}

\subsection{Recovery of \( \tau \)}
With the jump structure established, we now present the reconstruction procedure for \(\tau\). The key is that the jump in \(\partial_t u\) propagates to the boundary through the normal derivative, where it can be observed.
\begin{proof}[Proof of Theorem \ref{thm:mainthm} (Reconstruction of $\tau$)]
Since \( \partial_\nu \) commutes with time differentiation, the jump in \( \partial_t u \) at \( t = \tau \) propagates to the boundary, and we easily obtain the following corollary to Theorem \ref{thm:jump}:

\begin{corollary}[Boundary Jump Detection with Nonzero Boundary Data]
\[
\left[\partial_t(\partial_\nu u)\right]_{t=\tau^-}^{t=\tau^+}
= -\frac{c_2}{c_1} \partial_\nu \Phi_1(x) \quad \text{on } \partial\Omega.
\]
\end{corollary}
This corollary provides the link between the interior jump and the boundary measurement.

Making use of this corollary, we can detect the jump in \( \partial_t(\partial_\nu u) \) from the boundary measurement \( \partial_\nu u|_{\partial\Omega \times (0,T)} \). Let:
\[
\mathcal{S} = \{ t \in (0,T) : \partial_t(\partial_\nu u)(\cdot, t) \text{ has a jump} \}.
\]
Then:
\[
\tau = \min(\mathcal{S}).
\]
That is, the delay is identified as the first time at which a jump is observed in the time derivative of the Neumann boundary measurement.
\end{proof}

\subsection{Variations in Proof for Different Coefficient Regimes}
The jump structure described above depends on the specific values of the coefficients \(c_1, c_2, d_1, d_2\). We now analyze the modifications required when certain coefficients vanish, as this affects the order of the derivative at which the jump first appears.

\textbf{Case 1: \(c_2 = 0\) (No Time-Derivative Delay)}

When \(c_2 = 0\), by \eqref{eq:discontinuityform}, the jump in \(\partial_t u\) vanishes:
\[
\left[\partial_t u\right]_{t=\tau^-}^{t=\tau^+} = 0.
\]

In this scenario, the first nontrivial jump occurs in \(\partial_{tt}u\), which by \eqref{eq:eq:discontinuityformdtt}, is given by:
\[
\left[\partial_{tt} u\right]_{t=\tau^-}^{t=\tau^+}
= \frac{d_2}{c_1} \Delta \Phi_1(x) + \frac{q(x,\tau)}{c_1} \Phi_1(x).
\]

In the further subcase where \(d_2 = 0\), this simplifies to
\[
\left[\partial_{tt} u\right]_{t=\tau^-}^{t=\tau^+}
= \frac{q(x,\tau)}{c_1} \Phi_1(x).
\]

Hence, in this case, we must use the second order time derivative of the boundary measurement \(\partial_{tt} (\partial_\nu u)|_{\partial\Omega \times (0,T)} \) to reconstruct the time delay \(\tau\).
\vspace{1em}

\noindent\textbf{Case 2: Elliptic Case (\(c_1 = 0, d_1 > 0\))}
When the time derivative term in the principal part vanishes, the equation loses its parabolic character and becomes elliptic in space at each time. This case requires a different regularity framework and yields a jump in the Laplacian rather than in the time derivative.

In this case, we need to consider a more regular initial condition:
\begin{equation}\label{eq:ellipticregassump}
u_0 \in H^1(-\tau,0; H^2(\Omega)) \cap H^2(-\tau,0; L^2(\Omega)).
\end{equation}

We first recall the standard regularity result (c.f. \cite[Theorem 6.3.4]{EvansBook} or \cite[Chapter 3]{LadyzhenskayaUraltsevaBook}) for elliptic equations: Let \( \Omega' \subset \mathbb{R}^n \) be a bounded domain with \( C^2 \) boundary, \( d_1 > 0 \) be a constant and \( p \in L^\infty(\Omega') \). 

Consider the elliptic boundary value problem:
\[
\begin{cases}
- d_1 \Delta u(x) - p(x) u(x) = F(x), & \quad\text{in }\Omega, \\
u(x) = h(x), & \quad\text{in }\partial\Omega',
\end{cases}
\]

The standard elliptic regularity theory gives the following result:
\begin{theorem}[Elliptic Regularity]\label{thm:EllipThm}
If \( F \in L^2(\Omega') \) and \( h \in H^{3/2}(\partial\Omega') \), then the unique solution satisfies:
    \[
    u \in H^2(\Omega').
    \]
    Moreover, the following estimate holds:
    \[
    \|u\|_{H^2(\Omega')} \le C \left( \|F\|_{L^2(\Omega')} + \|h\|_{H^{3/2}(\partial\Omega')} \right).
    \]
\end{theorem}

When \(c_1 = 0\) but \(d_1 > 0\), the equation \eqref{eq:mainDPDE} becomes:
\begin{equation}\label{eq:EllipticCase}
c_2 \partial_t u(x,t-\tau) = d_1 \Delta u(x,t) + d_2 \Delta u(x,t-\tau) + p(x,t) u(x,t) + q(x,t) u(x,t-\tau).
\end{equation}

\textbf{Step 0:} On \(t \in [0,\tau]\), \(u(t-\tau) = u_0(t-\tau)\) and \(\partial_t u(t-\tau) = \partial_t u_0(t-\tau)\):
\begin{equation}\label{eq:EllipticCaseStep0}
d_1 \Delta u(x,t) + p(x,t) u(x,t) = c_2 \partial_t u_0(x,t-\tau) - d_2 \Delta u_0(x,t-\tau) - q(x,t) u_0(x,t-\tau).
\end{equation}
This is an elliptic equation for \(u(x,t)\) at each time \(t\). With Dirichlet boundary conditions and the regularity assumptions in \eqref{eq:mainregassump}, by Theorem \ref{thm:EllipThm}, the solution \(u^{(0)}\) is unique and 
\[u^{(0)}(\cdot,t)\in H^2(\Omega)\quad \text{ for each }t\in[0,\tau].\] Moreover, since the source term given by the right-hand-side of \eqref{eq:EllipticCaseStep0} is continuous in time by the regularity assumptions \eqref{eq:mainregassump} and \eqref{eq:ellipticregassump}, the solution \(u^{(0)}\) is continuous in time:
\[u^{(0)}\in C([0,\tau];H^2(\Omega)).\]

\textbf{Step 1:} On \(t \in [\tau, 2\tau]\):
\[
d_1 \Delta u(x,t) + p(x,t) u(x,t) = c_2 \partial_t u^{(0)}(x,t-\tau) - d_2 \Delta u^{(0)}(x,t-\tau) - q(x,t) u^{(0)}(x,t-\tau).
\]
Again, this is an elliptic equation with the same regularity on \([\tau,2\tau]\).

Taking the difference at \(t = \tau\):
\[
d_1 \Delta \left( u(x,\tau^+) - u(x,\tau^-) \right) = c_2 \Phi_1(x).
\]

Thus:
\begin{equation}\label{eq:EllipticCaseDeltaJumpForm}
\left[\Delta u\right]_{t=\tau^-}^{t=\tau^+} = \frac{c_2}{d_1} \Phi_1(x).
\end{equation}
In particular, if \(\Phi_1\not\equiv0\), then
\(\Delta u\) has a nontrivial jump at \(t=\tau\).
To detect this singularity from the boundary flux, we additionally
assume the boundary observability condition
\[
N(\Phi_1,0)\not\equiv0
\quad\text{on }\partial\Omega.
\]
Under this condition,
\[
[\partial_\nu u]_{\tau^-}^{\tau^+}
=
\frac{c_2}{d_1}N(\Phi_1,0)
\not\equiv0.
\]

\begin{remark}
The condition
\[
N(\Phi_1,0)\not\equiv0
\]
is an additional boundary observability condition. Indeed, a
nontrivial interior source does not necessarily produce a nontrivial
boundary normal derivative. For example, if
\(\varphi\in C_c^\infty(\Omega)\) is nonzero and
\(F=-d_1\Delta\varphi\), then the solution of
\[
-d_1\Delta w=F,\qquad w|_{\partial\Omega}=0,
\]
is \(w=\varphi\), for which
\[
F\not\equiv0,\qquad
\partial_\nu w|_{\partial\Omega}=0.
\]
Thus, the boundary observability condition cannot be omitted.
\end{remark}



\begin{remark}
It is instructive to contrast this elliptic case with the parabolic case. Although
\[
u^{(0)} \in C([0,\tau]; H^2(\Omega)).
\] and 
\[
u^{(1)} \in C([\tau, 2\tau]; H^2(\Omega)),
\]
at the interface \(t = \tau\), the source term may have a jump because the method of steps switches from using the initial history \(u_0\) to using the solution \(u^{(0)}\). Specifically,
\[
\left[F\right]_{t=\tau^-}^{t=\tau^+} = c_2 \Phi_1(x),
\]
where
\[
\Phi_1(x) = \partial_t u^{(0)}(x,0) - \partial_t u_0(x,0).
\]
Consequently, while \(u^{(0)},u^{(1)}\) are continuous in time on their respective time intervals, the Laplacian \(\Delta u\) has a jump at \(t = \tau\) if \(c_2 \Phi_1 \neq 0\). This is a characteristic feature of the elliptic regime.

This is because the main equation \eqref{eq:EllipticCase} is of an elliptic form. 
On the other hand, the general case in \eqref{eq:mainDPDE} is of parabolic form, and the standard parabolic result (Theorem \ref{thm:ParabolicThm} gives \(u\in C([0,T];H^1_0(\Omega))\) for \(u\) being the solution of \eqref{eq:mainDPDE}. 
\end{remark}
\vspace{1em}

\begin{table}[h]
\centering
\small
\begin{tabular}{|c|c|c|c|c|c|c|}
\hline
\textbf{Case} & \(c_1\) & \(c_2\) & \(d_1\) & \(d_2\) & \textbf{Type} & \textbf{First Jump} \\
\hline
General & \(>0\) & \(>0\) & \(\ge 0\) & \(\ge 0\) & Parabolic/ODE & \(\partial_t u\) at \(\tau\) \\
1 & \(>0\) & \(=0\) & \(\ge 0\) & \(\ge 0\) & Parabolic & \(\partial_{tt}u\) at \(\tau\) \\
2 & \(=0\) & \(>0\) & \(>0\) & \(\ge 0\) & Elliptic & \(\Delta u\) at \(\tau\) \\
\hline
\end{tabular}

\caption{Summary of coefficient regimes and their singularity structures.
The table indicates the first derivative (in time or space) that exhibits a jump at \(t=\tau\), which can be used to reconstruct the delay.}
\end{table}

\section{Recovery of \( p \) and \( q \)}\label{sec:recovercoef}
Having reconstructed the delay \(\tau\), we now address the recovery of the zeroth-order coefficients \(p\) and \(q\). As we will see, the structure of the problem on the initial interval \((0,\tau)\) provides a linear system for these coefficients, but a single measurement is insufficient to determine both uniquely.
On \( [0,\tau] \), from \eqref{eq:MainStep0}:
\[
    c_1 \partial_t u(x,t) + c_2 \partial_t u_0(x,t-\tau) = d_1 \Delta u(x,t) + d_2 \Delta u_0(x,t-\tau) + p(x,t) u(x,t) + q(x,t) u_0(x,t-\tau).
\]

Rearranging, we have:
\begin{equation}\label{eq:MainRecoverPQ}
p(x,t) u(x,t) + q(x,t) u_0(x,t-\tau) = H(x,t),
\end{equation}
where
\[
H(x,t) = c_1 \partial_t u(x,t) + c_2 \partial_t u_0(x,t-\tau) - d_1 \Delta u(x,t) - d_2 \Delta u_0(x,t-\tau).
\]



Thus, \eqref{eq:MainRecoverPQ} constitutes a single linear equation for two unknowns \( p(x,t) \) and \( q(x,t) \) at each time \( t \). To recover both coefficients, we require a second independent equation, which can be obtained from a second experiment for the same set of unknown coefficients \(p,q\).

\subsection{The Non-Uniqueness Construction}
Before presenting the stability result, we first demonstrate that a single measurement is insufficient in the general case where $p,q$ depends on both $x$ and $t$. The following construction shows that there is a gauge freedom in the coefficients: one can add a term to \(p\) and subtract a corresponding term from \(q\) without changing the solution or the boundary measurement.

\begin{theorem}[Non-Uniqueness with a Single Measurement]
Let \((p, q)\) be a coefficient pair that produces the solution \(u\) and the boundary measurement \(\partial_\nu u|_{\Sigma_T}\). Let \(r(x,t)\) be any function such that:
\[
r \in C^1([0,T]; L^\infty(\Omega)), \qquad r(x,t) \neq 0 \text{ on a set of positive measure}.
\]

Define:
\[
p'(x,t) = p(x,t) + r(x,t) u_0(x,t-\tau),
\]
\[
q'(x,t) = q(x,t) - r(x,t) u(x,t).
\]

Then \((p', q') \neq (p, q)\) in general, but \((p', q')\) produces the same solution \(u\) and hence the same boundary measurement \(\partial_\nu u|_{\Sigma_T}\).
\end{theorem}

\begin{proof}
We verify that \((p', q')\) satisfies the same equation \eqref{eq:MainRecoverPQ} as \((p, q)\).

For \((p, q)\), we have:
\[
p u + q u_0 = H.
\]

For \((p', q')\), we compute:
\[
\begin{aligned}
p' u + q' u_0
&= \bigl( p + r u_0 \bigr) u + \bigl( q - r u \bigr) u_0 \\
&= p u + r u_0 u + q u_0 - r u u_0 \\
&= p u + q u_0 + r u u_0 - r u u_0 \\
&= p u + q u_0 \\
&= H.
\end{aligned}
\]

Thus \((p', q')\) satisfies the same equation \eqref{eq:MainRecoverPQ} with the same solution \(u\). Since the solution \(u\) is unchanged, the boundary measurement \(\partial_\nu u|_{\Sigma_T}\) is also unchanged.

However, if \(r\) is not identically zero, then:
\[
p' - p = r u_0 \neq 0, \qquad q' - q = -r u \neq 0.
\]

So \((p', q') \neq (p, q)\). Therefore, a single measurement does not determine the coefficients uniquely.
\end{proof}

\begin{remark}
    A concrete construction for \(r\) is as follows: Let \(\Omega = (0,1)\), \(\tau = 1\), and suppose we have a solution \(u(x,t)\) with \(u_0(x,t) \neq 0\) and \(u(x,t) \neq 0\) for \((x,t) \in \Omega \times (0,\tau)\). Choose:
\[
r(x,t) = \varepsilon \sin(\pi x) \cos(\pi t),
\]
where \(\varepsilon > 0\) is small.
Then:
\[
p'(x,t) = p(x,t) + \varepsilon \sin(\pi x) \cos(\pi t) u_0(x,t-\tau),
\]
\[
q'(x,t) = q(x,t) - \varepsilon \sin(\pi x) \cos(\pi t) u(x,t).
\]

Then, one can easily verify that \((p', q') \neq (p, q)\) for \(\varepsilon \neq 0\), but both pairs produce the same solution \(u\) and the same boundary measurement.
\end{remark}


To overcome the non-uniqueness shown above, we need extra assumptions, and a second independent measurement.
We now turn to a refined stability analysis for the coefficient \(p(x),q(x)\) is known. The key idea is to use the initial layer to derive an explicit relation between the unknown coefficient difference and the initial trace of the time derivative of the solution difference. This relation, combined with a Carleman estimate, yields a Lipschitz stability bound in terms of the boundary measurement.

\subsection{Proof of Theorem}

\subsubsection{Step 1: Reduction to the initial layer}

Set
\[
\pi:=p_1-p_2,
\qquad
\chi:=q_1-q_2,
\]
and, for each experiment $i=1,2$, define 
\[v^{[i]}:=u^{[i]}_1-u^{[i]}_2.\] Since, within each experiment, the two solutions correspond to the same prescribed history and the same Dirichlet boundary data, we have
\begin{align}
v^{[i]}(x,t)=0
&\quad\text{on }\partial\Omega\times(-\tau,T),
\label{eq:v-boundary}\\
v^{[i]}(x,t)=0
&\quad\text{in }\Omega\times[-\tau,0].
\label{eq:v-history}
\end{align}
Subtracting the two equations corresponding to $u^{[i]}_1$ and $u^{[i]}_2$ gives
\begin{equation}
\label{eq:v-full}
\begin{aligned}
    c_1\partial_tv^{[i]}
    +c_2\partial_tv^{[i]}(x,t-\tau)
    &=d_1\Delta v^{[i]}
    +d_2\Delta v^{[i]}(x,t-\tau)
    \\&\quad+p_1v^{[i]}
    +q_1v^{[i]}(x,t-\tau)
    \\&\quad+\pi(x)u_2^{[i]}(x,t)+\chi(x)u_2^{[i]}(x,t-\tau).
\end{aligned}
\end{equation}
This difference equation holds for all $t \in (0,T)$. On the initial layer $(0,\tau)$, the shifted argument $t-\tau$ falls within the prescribed historical interval $(-\tau,0)$, so $u_2^{[i]}(x,t-\tau) = u_0^{[i]}(x,t-\tau)$ and all delayed terms of $v^{[i]}$ vanish identically. For $0<t<\tau$, one has $t-\tau\in(-\tau,0)$. Hence, by \eqref{eq:v-history}--\eqref{eq:v-boundary},
\[
v^{[i]}(x,t-\tau)=0,
\]
hence
\[\partial_t v^{[i]}(x,t-\tau)=0, \qquad
\Delta v(x,t-\tau)=0.
\]
This makes sense by the regularity of $u^{[i]}\in C([0,\tau]; H^1_0(\Omega))$ at the initial step $(0,\tau)$ by the standard parabolic theory given in Theorem \ref{thm:ParabolicThm}. 
Moreover, 
\[u_2^{[i]}(x,t-\tau)=u_0^{[i]}(x,t-\tau)\]
Therefore, on the initial layer $Q_\tau$ , the delayed terms disappear identically, and \eqref{eq:v-full} reduces to
\begin{equation}
\label{eq:v-initial-layer}
c_1\partial_tv^{[i]}-d_1\Delta v^{[i]}-p_1v
=
\pi u_2^{[i]}+\chi u_0^{[i]}(x,t-\tau)
\qquad\text{in }Q_\tau.
\end{equation}
This reduction to a non-delayed equation is the crucial simplification that makes the subsequent Carleman estimate applicable. We emphasize that no restriction $c_2,d_2 = 0$ is required in this reduction. The delayed difference terms vanish identically on the initial layer because the two coefficient configurations share the same prescribed history. Next, define
\[
w^{[i]}:=\partial_tv^{[i]}.
\]
Since $p_1=p_1(x)$, $\pi=\pi(x)$ and $\chi=\chi(x)$ are independent of time, differentiating
\eqref{eq:v-initial-layer} yields
\begin{equation}
\label{eq:w-equation}
c_1\partial_tw^{[i]}-d_1\Delta w^{[i]}-p_1w^{[i]}
=
\pi \partial_tu^{[i]}_2+\chi\partial_tu_0^{[i]}(x,t-\tau)
\qquad\text{in }Q_\tau.
\end{equation}

We now derive an algebraic relation between the coefficient differences and the initial traces $w^{[i]}(\cdot,0)$. Because $v^{[i]}=0$ on $\partial\Omega\times(0,\tau)$ and the boundary data are time-independent with respect to the coefficient,
\begin{equation}\label{eq:tracew}
    w^{[i]}=0\qquad\text{on }\Sigma_\tau,
\end{equation}
the assumed regularity implies
\begin{equation}\label{eq:traceLapw}
    \Delta w^{[i]}(\cdot,0)=0.
\end{equation}
Taking the right-hand trace at $t=0$ in \eqref{eq:v-full}, and using
\eqref{eq:tracew}--\eqref{eq:traceLapw}, yields
\[
c_1w^{[i]}(x,0)
=
\pi(x)u_2^{[i]}(x,0)+\chi(x) u_0^{[i]}(x,-\tau)
\]

Since
\[
u_2^{[i]}(x,0)=u_0^{[i]}(x,0),
\]
we obtain
\begin{equation}
c_1
\begin{pmatrix}
w^{[1]}(x,0)\\
w^{[2]}(x,0)
\end{pmatrix}
=
A_0(x)
\begin{pmatrix}
\pi(x)\\
\chi(x)
\end{pmatrix}.
\end{equation}

By assumption \eqref{eq:pqnondeg},
\[
|A_0(x)\xi|
\ge
\mu_0|\xi|,
\qquad
\xi\in\mathbb{R}^2.
\]
Consequently,
\[
c_1^2
\sum_{i=1}^2
|w^{[i]}(x,0)|^2
\ge
\mu_0^2
\left(
|\pi(x)|^2+|\chi(x)|^2
\right),
\tag{4.29}
\]
and hence
\begin{equation}\label{eq:r-w0}
|\pi(x)|^2+|\chi(x)|^2
\le
\frac{c_1^2}{\mu_0^2}
\sum_{i=1}^2
|w^{[i]}(x,0)|^2.
\end{equation}


Thus, to obtain the desired stability result \eqref{eq:final-stability}, it remains to estimate the initial trace of $w^{[i]}$ by the measured boundary data.

\begin{remark}[Necessity of two experiments for the simultaneous recovery of $p$ and $q$]
The requirement of two experiments arises precisely from the algebraic
reconstruction of the two unknown coefficient differences
\[
\pi:=p_1-p_2,
\qquad
\chi:=q_1-q_2.
\]
Indeed, for a single experiment, the initial-layer equation yields at
$t=0$ the relation
\[
c_1 w(x,0)
=
\pi(x)u_0(x,0)
+
\chi(x)u_0(x,-\tau).
\]
Thus, at each fixed $x\in\Omega$, one obtains only one scalar equation
for the two unknown quantities $\pi(x)$ and $\chi(x)$. Consequently, even if
the Carleman estimate determines $w(\cdot,0)$ exactly, it determines only
the single linear combination
\[
\pi(x)u_0(x,0)+\chi(x)u_0(x,-\tau),
\]
and, in general, it cannot determine $\pi(x)$ and $\chi(x)$ separately.

More explicitly, set
\[
a(x):=u_0(x,0),
\qquad
b(x):=u_0(x,-\tau).
\]
If $(a(x),b(x))\neq(0,0)$, then the nonzero pair
\[
\bigl(\pi(x),\chi(x)\bigr)=\bigl(b(x),-a(x)\bigr)
\]
satisfies
\[
a(x)\pi(x)+b(x)\chi(x)=0.
\]
Hence, for this nonzero pair,
\[
w(x,0)=0,
\]
while
\[
|\pi(x)|^2+|\chi(x)|^2
=
|a(x)|^2+|b(x)|^2>0.
\]
Therefore, no estimate of the form
\[
|\pi(x)|^2+|\chi(x)|^2
\leq C|w(x,0)|^2
\]
can hold in general for a single experiment, regardless of the
Carleman estimate used to control $w(\cdot,0)$.

With two experiments, on the other hand, the corresponding initial-time
relations form the $2\times2$ linear system
\[
c_1
\begin{pmatrix}
w^{[1]}(x,0)\\[1mm]
w^{[2]}(x,0)
\end{pmatrix}
=
A_0(x)
\begin{pmatrix}
\pi(x)\\[1mm]
\chi(x)
\end{pmatrix},
\]
where
\[
A_0(x):=
\begin{pmatrix}
u_0^{[1]}(x,0) & u_0^{[1]}(x,-\tau)\\
u_0^{[2]}(x,0) & u_0^{[2]}(x,-\tau)
\end{pmatrix}.
\]
Under the uniform nondegeneracy assumption
\[
|A_0(x)\xi|\geq\mu_0|\xi|,
\qquad
x\in\overline{\Omega},\quad
\xi\in\mathbb R^2,
\]
we obtain
\[
\begin{aligned}
\mu_0^2\bigl(|\pi(x)|^2+|\chi(x)|^2\bigr)
&\leq
\left|
A_0(x)
\begin{pmatrix}
\pi(x)\\[1mm]
\chi(x)
\end{pmatrix}
\right|^2\\
&=
c_1^2
\left(
|w^{[1]}(x,0)|^2+
|w^{[2]}(x,0)|^2
\right).
\end{aligned}
\]
Consequently,
\[
|\pi(x)|^2+|\chi(x)|^2
\leq
\frac{c_1^2}{\mu_0^2}
\sum_{i=1}^2
|w^{[i]}(x,0)|^2.
\]

Thus, the need for two experiments is not a consequence of the
Carleman estimate itself. The Carleman estimate is applied separately
to each experiment to estimate the corresponding initial trace
$w^{[i]}(\cdot,0)$. The use of two experiments is required at the
subsequent algebraic reconstruction step, where two independent scalar
relations are needed to distinguish the two unknown coefficient
differences $\pi$ and $\chi$.

If one of the two coefficients is known a priori, this obstruction
disappears. For example, if $q_1=q_2$, so that $\chi=0$, then a single
experiment gives
\[
c_1w(x,0)=\pi(x)u_0(x,0),
\]
and the usual nondegeneracy condition
\[
|u_0(x,0)|\geq\kappa_0>0
\]
implies
\[
|\pi(x)|^2
\leq
\frac{c_1^2}{\kappa_0^2}|w(x,0)|^2.
\]
Thus, one experiment is sufficient for the recovery of a single
unknown coefficient, whereas two experiments are, in general,
necessary for the simultaneous Lipschitz recovery of the two
independent coefficients $p$ and $q$.
\end{remark}

\subsubsection{Step 2: Carleman Estimate}
We now establish a Carleman estimate for the non-delayed operator \(P_p = c_1\partial_t - d_1\Delta - p(x)\). This estimate provides a weighted energy inequality that controls the initial trace of \(w\) in terms of the source term and the Neumann boundary data.

Choose
\(
x_0\in\mathbb R^n\setminus\overline\Omega
\)
and define
\[
\psi(x):=|x-x_0|^2.
\]
Set
\[
\psi_-:=\min_{\overline\Omega}\psi,
\qquad
\psi_+:=\max_{\overline\Omega}\psi.
\]
Since $x_0\notin\overline\Omega$,
\begin{equation}\label{eq:grad-psi-lower}
|\nabla\psi(x)|=2|x-x_0|
\ge 2\rho_0,\end{equation}
where 
\(\rho_0:=\operatorname{dist}(x_0,\overline\Omega)>0.
\)

For $\lambda\ge1$, set
\begin{equation}\label{eq:psi}
\Psi(x):=e^{\lambda\psi(x)}.
\end{equation}
Then
\begin{align}
\nabla\Psi
&=
\lambda e^{\lambda\psi}\nabla\psi,
\label{eq:grad-psi-phi}\\
D^2\Psi
&=
e^{\lambda\psi}
\left(
2\lambda I_n
+
\lambda^2\nabla\psi\otimes\nabla\psi
\right),\\
\Delta\Psi
&=
e^{\lambda\psi}
\left(
2n\lambda
+
\lambda^2|\nabla\psi|^2
\right).
\end{align}

In particular, 
\begin{equation}
\label{eq:grad-lower}
|\nabla\Psi|
\ge
2\rho_0\lambda e^{\lambda\psi},
\end{equation}
and, for any $\zeta\in\mathbb{R}^n$, the Hessian satisfies 
\begin{equation}
\label{eq:hessian-lower}
D^2\Psi[\zeta,\zeta] = \zeta^T(D^2\Psi)\zeta=e^{\lambda\psi}\left(2\lambda |\zeta|^2+\lambda^2(\nabla\psi\cdot\zeta)^2\right)
\ge
2\lambda e^{\lambda\psi}|\zeta|^2.
\end{equation}

We now choose $\kappa>0$ so that
\begin{equation}
\label{eq:kappa-choice}
\kappa\tau>\Psi_+-\Psi_-,
\end{equation}
where
\[
\Psi_-:=e^{\lambda\psi_-},
\qquad
\Psi_+:=e^{\lambda\psi_+}.
\]
Such a $\kappa$ exists for every fixed $\tau>0$.

Define
\begin{equation}
\label{eq:phi}
\varphi(x,t):=\kappa t-\Psi(x).
\end{equation}
Then
\begin{equation}\label{eq:phi-Psi-identity}
\varphi_t=\kappa,
\qquad
\nabla\varphi=-\nabla\Psi,
\qquad
\Delta\varphi=-\Delta\Psi.
\end{equation}

Furthermore,
\begin{equation}
\label{eq:phi-difference}
\varphi(x,\tau)=\varphi(x,0)+\kappa\tau.
\end{equation}

Then, the following Carleman estimate holds:
\begin{lemma}[Initial-trace Carleman estimate]
\label{lem:carleman}
Define
\[
P_pz:=c_1z_t-d_1\Delta z-p(x)z,
\]
where
\[
\|p\|_{W^{1,\infty}(\Omega)}\le M.
\]
Let $\Psi$ and $\varphi$ be defined by
\eqref{eq:psi} and \eqref{eq:phi} respectively, with $\lambda\ge1$ and $\kappa>0$ fixed.

Then, there exist constants
\[
\lambda_0\ge1,
\qquad
s_0=s_0(\lambda,\kappa,M,\Omega,c_1,d_1)\ge1,
\qquad
C>0,
\]
uniform with respect to all $p\in\mathcal P_M$, 
such that, for every $\lambda\ge\lambda_0$, every $s\ge s_0$, and every
\[
z\in H^1(0,\tau;L^2(\Omega))
\cap L^2(0,\tau;H^2(\Omega))
\]
with
\[
z=0
\qquad\text{on }\Sigma_\tau,
\]
one has
\begin{equation}
\begin{aligned}
&
s\lambda
\int_{Q_\tau}
e^{-2s\varphi}|\nabla z|^2\,dx\,dt
+
s^3\lambda^4
\int_{Q_\tau}
e^{-2s\varphi}|z|^2\,dx\,dt
+
s^2\lambda^2
\int_\Omega
e^{-2s\varphi(x,0)}
|z(x,0)|^2\,dx\\
&\le
C
\int_{Q_\tau}
e^{-2s\varphi}
|P_pz|^2\,dx\,dt
+
Cs\lambda e^{\lambda\psi_+}
\int_{\Sigma_\tau}
e^{-2s\varphi}
|\partial_\nu z|^2\,d\sigma\,dt
+
Cs^2\lambda^2
\int_\Omega
e^{-2s\varphi(x,\tau)}
|z(x,\tau)|^2\,dx.
\label{eq:carleman}
\end{aligned}
\end{equation}
\end{lemma}

\begin{proof}
The proof follows the standard method for Carleman estimates \cite{BukhgeimKlibanov1981Carleman,Klibanov1992IPCarleman,Yamamoto2009IPCarlemanParabolic}: conjugate the operator with the exponential weight \(e^{s\varphi}\), integrate by parts, and collect the positive definite terms. The spatial weight \(\Psi\) is chosen to ensure coercivity of the principal terms.

Set
\[
y=e^{-s\varphi}z,
\qquad
z=e^{s\varphi}y.
\]
A direct differentiation gives
\[
e^{-s\varphi}z_t
=
y_t+s\varphi_ty
\]
and
\[
e^{-s\varphi}\Delta z
=
\Delta y
+
2s\nabla\varphi\cdot\nabla y
+
s(\Delta\varphi)\,y
+
s^2|\nabla\varphi|^2y.
\]
Consequently,
\begin{equation}
e^{-s\varphi}P_pz
=
c_1y_t-d_1\Delta y
-2d_1s\nabla\varphi\cdot\nabla y+
\left(
c_1s\varphi_t
-d_1s(\Delta\varphi)
-d_1s^2|\nabla\varphi|^2
-p
\right)y.
\label{eq:conjugated}
\end{equation}

Write
\begin{align*}
A&:=-d_1\Delta y,
&
B&:=-d_1s^2|\nabla\varphi|^2y,\\
C&:=c_1s\kappa y,
&
D&:=c_1y_t,\\
E&:=-2d_1s\nabla\varphi\cdot\nabla y,
&
F&:=-d_1s(\Delta\varphi)y,\\
G&:=-py.
\end{align*}
Then, by \eqref{eq:phi-Psi-identity},
\[
e^{-s\varphi}P_pz=A+B+C+D+E+F+G.
\]

We estimate
\[
\int_{Q_\tau}|A+B+C+D+E+F+G|^2\,dx\,dt.
\]
We compute the cross terms that generate the principal coercive contributions (highest powers of $s$) by integration by parts. All remaining cross terms are estimated by Cauchy-Schwarz and Young's inequalities and absorbed into these coercive terms for sufficiently large $s$ and $\lambda$.

First, by \eqref{eq:phi-Psi-identity},
\begin{equation}
2\int_{Q_\tau}AE\,dx\,dt
=
-4d_1^2s
\int_{Q_\tau}
\Delta y\,\nabla\Psi\cdot\nabla y\,dx\,dt.
\end{equation}
Since $y=0$ on $\Sigma_\tau$, integration by parts gives
\begin{multline}
2\int_{Q_\tau}AE\,dx\,dt
\\=
4d_1^2s
\int_{Q_\tau}
D^2\Psi[\nabla y,\nabla y]
\,dx\,dt
-
2d_1^2s
\int_{Q_\tau}
\Delta\Psi|\nabla y|^2
\,dx\,dt-
2d_1^2s
\int_{\Sigma_\tau}
\partial_\nu\Psi|\partial_\nu y|^2
\,d\sigma\,dt.
\label{eq:AE}
\end{multline}

Next, again by \eqref{eq:phi-Psi-identity},
\begin{equation}
2\int_{Q_\tau}AF\,dx\,dt
=
-2d_1^2s
\int_{Q_\tau}
\Delta y\,\Delta\Psi\,y
\,dx\,dt.
\end{equation}
Another integration by parts, using $y=0$ on the lateral boundary, yields
\begin{equation}
2\int_{Q_\tau}AF\,dx\,dt
=
2d_1^2s
\int_{Q_\tau}
\Delta\Psi|\nabla y|^2
\,dx\,dt-
d_1^2s
\int_{Q_\tau}
\Delta^2\Psi\,y^2
\,dx\,dt.
\label{eq:AF}
\end{equation}
Adding \eqref{eq:AE} and \eqref{eq:AF},
\begin{multline}
2\int_{Q_\tau}(AE+AF)\,dx\,dt
\\=
4d_1^2s
\int_{Q_\tau}
D^2\Psi[\nabla y,\nabla y]
\,dx\,dt
-
d_1^2s
\int_{Q_\tau}
\Delta^2\Psi\,y^2
\,dx\,dt
-
2d_1^2s
\int_{\Sigma_\tau}
\partial_\nu\Psi|\partial_\nu y|^2
\,d\sigma\,dt.
\label{eq:gradient-identity}
\end{multline}

The terms with the highest powers of $s$ are obtained from $BE$ and $BF$:
\begin{equation}
2\int_{Q_\tau}BE\,dx\,dt
=
-4d_1^2s^3
\int_{Q_\tau}
|\nabla\Psi|^2
\nabla\Psi\cdot y\nabla y
\,dx\,dt
=
2d_1^2s^3
\int_{Q_\tau}
\operatorname{div}
\left(
|\nabla\Psi|^2\nabla\Psi
\right)
y^2
\,dx\,dt,
\end{equation}
\[
2\int_{Q_\tau}BF\,dx\,dt
=
-2d_1^2s^3
\int_{Q_\tau}
|\nabla\Psi|^2\Delta\Psi\,y^2
\,dx\,dt.
\]
Using
\[
\operatorname{div}
\left(
|\nabla\Psi|^2\nabla\Psi
\right)
=
2D^2\Psi[\nabla\Psi,\nabla\Psi]
+
|\nabla\Psi|^2\Delta\Psi,
\]
we obtain 
\begin{equation}
2\int_{Q_\tau}(BE+BF)\,dx\,dt
=
4d_1^2s^3
\int_{Q_\tau}
D^2\Psi[\nabla\Psi,\nabla\Psi]y^2
\,dx\,dt.
\label{eq:zero-identity}
\end{equation}

Finally,
\begin{equation}
\begin{aligned}
2\int_{Q_\tau}BD\,dx\,dt
&=
-2c_1d_1s^2
\int_{Q_\tau}
|\nabla\Psi|^2yy_t\,dx\,dt
\\&=
c_1d_1s^2
\int_\Omega
|\nabla\Psi|^2|y(x,0)|^2\,dx
-
c_1d_1s^2
\int_\Omega
|\nabla\Psi|^2|y(x,\tau)|^2\,dx.
\end{aligned}
\label{eq:initial-identity}
\end{equation}

We now estimate the terms on the right-hand-side.

For a fixed $x_0$ and $\lambda\ge1$, since
\[
\psi(x)=|x-x_0|^2,
\]
for every $1\le k\le4$, there exists $C_k>0$ such that
\begin{equation}
\label{eq:weight-derivatives}
|D^k\Psi|
\le
C_k\lambda^k e^{\lambda\psi}.
\end{equation}
In particular, every individual component of the fourth-order derivative tensor of $\Psi$
\[
D^4\Psi
=
\left(
\partial_{ijkl}\Psi
\right)_{1\le i,j,k,l\le n}
\]
is bounded by its Euclidean
tensor norm:
\[
\left|\partial_{ijkl}\Psi\right|
\le
\left|D^4\Psi\right|.
\]
Since
\[
\Delta^2\Psi
=
\Delta(\Delta\Psi)
=
\sum_{j=1}^n\partial_{jj}
\left(
\sum_{i=1}^n\partial_{ii}\Psi
\right)
=
\sum_{i,j=1}^n
\partial_{jjii}\Psi,
\]
by the triangle inequality,
\[
\left|\Delta^2\Psi\right|
\le
\sum_{i,j=1}^n
\left|\partial_{jjii}\Psi\right|
\le
\sum_{i,j=1}^n
\left|D^4\Psi\right|=
n^2\left|D^4\Psi\right|.
\]
Consequently, 
\[
\left|\Delta^2\Psi\right|
\le
n^2C_4\lambda^4e^{\lambda\psi}=C\lambda^4e^{\lambda\psi}
\]
after redefining the constant.

On the other hand, by \eqref{eq:grad-psi-lower} and \eqref{eq:grad-psi-phi},
\begin{equation}
D^2\Psi[\nabla\Psi,\nabla\Psi]
=
e^{\lambda\psi}
\left(
2\lambda|\nabla\Psi|^2
+
\lambda^2|\nabla\psi\cdot\nabla\Psi|^2
\right)
\ge
\lambda^4e^{3\lambda\psi}|\nabla\psi|^4
\ge
16\rho_0^4\lambda^4e^{3\lambda\psi}.
\label{eq:zero-coercivity}
\end{equation}

Consequently, by \eqref{eq:grad-lower} and \eqref{eq:hessian-lower}, the three principal terms in \eqref{eq:gradient-identity}, \eqref{eq:zero-identity} and \eqref{eq:initial-identity} satisfy
\begin{align}
4d_1^2s
\int_{Q_\tau}
D^2\Psi[\nabla y,\nabla y]\,dx\,dt
&\ge
8d_1^2s\lambda
\int_{Q_\tau}
e^{\lambda\psi}|\nabla y|^2\,dx\,dt,
\label{eq:principal-grad}\\
4d_1^2s^3
\int_{Q_\tau}
D^2\Psi[\nabla\Psi,\nabla\Psi]y^2\,dx\,dt
&\ge
64d_1^2\rho_0^4s^3\lambda^4
\int_{Q_\tau}
e^{3\lambda\psi}y^2\,dx\,dt,
\label{eq:principal-zero}\\
c_1d_1s^2
\int_\Omega
|\nabla\Psi|^2|y(x,0)|^2\,dx
&\ge
4c_1d_1\rho_0^2s^2\lambda^2
\int_\Omega
e^{2\lambda\psi}|y(x,0)|^2\,dx.
\label{eq:principal-initial}
\end{align}

The negative term containing $\Delta^2\Psi$ in \eqref{eq:gradient-identity} is estimated by
\[
d_1^2s
\int_{Q_\tau}
|\Delta^2\Psi|y^2\,dx\,dt
\le
Cd_1^2s\lambda^4
\int_{Q_\tau}
e^{\lambda\psi}y^2\,dx\,dt.
\]
By comparison with \eqref{eq:principal-zero}, this is absorbed once $s\ge s_0$ for $s_0$ sufficiently large, since $e^{\lambda\psi}\le e^{3\lambda\psi}$.

All terms involving $p$ are lower order. Indeed, since
\[
|p|\le M,
\qquad
|\nabla p|\le M,
\]
we have, for example,
\begin{align}
2\left|\int_{Q_\tau}AG\,dx\,dt\right|
&\le
C_M\int_{Q_\tau}|\nabla y|^2\,dx\,dt
+
C_M\int_{Q_\tau}|y|^2\,dx\,dt,\\
2\left|\int_{Q_\tau}BG\,dx\,dt\right|
&\le
C_Ms^2\lambda^2
\int_{Q_\tau}
e^{2\lambda\psi}y^2\,dx\,dt,\\
2\left|\int_{Q_\tau}EG\,dx\,dt\right|
&\le
\varepsilon s\lambda
\int_{Q_\tau}
e^{\lambda\psi}|\nabla y|^2\,dx\,dt
+
C_{\varepsilon,M}s\lambda^3
\int_{Q_\tau}
e^{3\lambda\psi}y^2\,dx\,dt.
\end{align}
The first term is absorbed by \eqref{eq:principal-grad}, while the second terms are absorbed by \eqref{eq:principal-zero} by taking $\lambda$ and then $s$ sufficiently large.

The terms containing $C=c_1s\kappa y$ are also lower order in $s$ relative to the $s^3$ coercivity. For example,
\[
2|BC|
\le
\varepsilon |B|^2
+
C_\varepsilon |C|^2,
\]
and
\[
|C|^2
\le
C s^2\kappa^2y^2.
\]
Since $\kappa$ is fixed before $s$ is chosen, this is absorbed by the $s^3$ zero-order term for sufficiently large $s$.

Likewise,
\[
2|CE|
\le
\varepsilon s\lambda e^{\lambda\psi}|\nabla y|^2
+
C_{\varepsilon,\kappa}s^3\lambda^3e^{3\lambda\psi}y^2,
\]
and the two resulting terms are absorbed by the principal gradient and zero-order terms.

The terms involving $D=c_1y_t$ are treated by retaining a fixed fraction of the positive square
\[
\int_{Q_\tau}|D|^2\,dx\,dt
=
c_1^2\int_{Q_\tau}|y_t|^2\,dx\,dt
\]
on the right-hand side. In particular, the elementary inequality
\[
2|XY|
\le
\varepsilon X^2+\varepsilon^{-1}Y^2
\]
is applied to every mixed term containing $D$. Thus, choosing $\varepsilon$ sufficiently small, no uncontrolled $y_t$ term remains on the left-hand side.

The negative terminal contribution in \eqref{eq:initial-identity}, together with the terminal contributions generated by the terms containing $C$ and $D$, is bounded by
\begin{equation}
\label{eq:terminal-carleman-bound}
Cs^2\lambda^2
\int_\Omega
e^{-2s\varphi(x,\tau)}
|z(x,\tau)|^2\,dx.
\end{equation}

Finally, on $\Sigma_\tau$,
\[
\partial_\nu y
=
e^{-s\varphi}\partial_\nu z
\]
because $z=0$ on $\Sigma_\tau$. Also,
\[
|\partial_\nu\Psi|
\le
|\nabla\Psi|
\le
C\lambda e^{\lambda\psi_+}.
\]
Therefore the boundary term in \eqref{eq:gradient-identity} is bounded by
\[
Cs\lambda e^{\lambda\psi_+}
\int_{\Sigma_\tau}
e^{-2s\varphi}
|\partial_\nu z|^2\,d\sigma\,dt.
\]

Collecting all the estimates and choosing first $\lambda\ge\lambda_0$ (so that we can have a lower bound for terms of the form $e^{k\lambda\psi}>0$) and then $s\ge s_0(\lambda)$ sufficiently large yields
\eqref{eq:carleman}.

\end{proof}

\subsubsection{Step 3: Absorption of the terminal term}

The Carleman estimate \eqref{eq:carleman} contains a terminal term at \(t=\tau\) that cannot be discarded directly. However, using an energy estimate for \(w\), we can show that this terminal term is controlled by the initial trace, which is precisely what we need to estimate.

Recall that $w^{[i]}$ satisfies
\[c_1\partial_tw^{[i]}-d_1\Delta w^{[i]}-p_1w^{[i]}
=
\pi \partial_tu^{[i]}_2+\chi\partial_tu_0^{[i]}(x,t-\tau)\]
Multiplying by $w^{[i]}$ and integrating over $\Omega$, using
\[
w^{[i]}=0
\qquad\text{on }\partial\Omega,
\]
gives, for each fixed $t$,
\begin{equation}
\frac{c_1}{2}\frac{d}{dt}\|w^{[i]}(t)\|_{L^2(\Omega)}^2
+
d_1\|\nabla w^{[i]}(t)\|_{L^2(\Omega)}^2
=
\int_\Omega p_1|w^{[i]}|^2\,dx
+
\int_\Omega \left[\pi \partial_tu^{[i]}_2+\chi\partial_tu_0^{[i]}(x,t-\tau)\right]w^{[i]}\,dx.
\end{equation}
By the a priori bounds \eqref{eq:pq-bound} and \eqref{eq:ut-bound}, and Young's inequality,
we obtain
\begin{equation}
\begin{aligned}
\frac{c_1}{2}\frac{d}{dt}\|w^{[i]}(t)\|_{L^2}^2
+
d_1\|\nabla w^{[i]}(t)\|_{L^2}^2
&\le
M\|w^{[i]}(t)\|_{L^2}^2
+
M_1\|\pi\|_{L^2}\|w^{[i]}(t)\|_{L^2}
+ 
M_1\|\chi\|_{L^2}\|w^{[i]}(t)\|_{L^2}\\
&\le
\left(
M+\frac{2M_1^2}{c_1}
\right)
\|w^{[i]}(t)\|_{L^2}^2
+
\frac{c_1}{4}\|\pi\|_{L^2}^2
+
\frac{c_1}{4}\|\chi\|_{L^2}^2
\end{aligned}
\end{equation}
for each fixed $t$. 
Since 
\[\frac{c_1}{2}\frac{d}{dt}\|w^{[i]}(t)\|_{L^2}^2
+
d_1\|\nabla w^{[i]}(t)\|_{L^2}^2
\ge \frac{c_1}{2}\frac{d}{dt}\|w^{[i]}(t)\|_{L^2}^2,
\]
by Grönwall's inequality,
\[
\|w^{[i]}(t)\|_{L^2(\Omega)}^2
\le
C_E
\left(
\|w^{[i]}(\cdot,0)\|_{L^2(\Omega)}^2
+
\|\pi\|_{L^2(\Omega)}^2
+
\|\chi\|_{L^2(\Omega)}^2
\right),
\qquad 0\le t\le\tau.
\]
Consequently,
\begin{equation}
\label{eq:w-energy}
\sup_{0\le t\le\tau}
\|w^{[i]}(t)\|_{L^2(\Omega)}^2
\le
C_E
\left(
\|w^{[i]}(\cdot,0)\|_{L^2(\Omega)}^2
+
\|\pi\|_{L^2(\Omega)}^2
+
\|\chi\|_{L^2(\Omega)}^2
\right),
\end{equation}
where
\[
C_E=C_E(c_1,M,M_1,\tau).
\]

Using \eqref{eq:r-w0},
\[\|\pi\|_{L^2(\Omega)}^2+\|\chi\|_{L^2(\Omega)}^2
\le
\frac{c_1^2}{\mu_0^2}
\sum_{i=1}^2
\|w^{[i]}(\cdot,0)\|_{L^2(\Omega)}^2.\]
together with the uniform boundedness of the prescribed histories, we obtain
\[
\begin{aligned}
\sup_{0\le t\le\tau}
\sum_{i=1}^2\|w^{[i]}(t)\|_{L^2(\Omega)}^2
&\le
C_E
\left(
1+\frac{c_1^2}{\mu_0^2}
\right)
\sum_{i=1}^2\|w^{[i]}(\cdot,0)\|_{L^2(\Omega)}^2.
\end{aligned}
\]
In particular, taking $t=\tau$ gives
\[
\sum_{i=1}^2\|w^{[i]}(\cdot,\tau)\|_{L^2(\Omega)}^2
\le
C_E
\left(
1+\frac{c_1^2}{\mu_0^2}
\right)
\sum_{i=1}^2\|w^{[i]}(\cdot,0)\|_{L^2(\Omega)}^2.
\]
Thus, defining
\[
C_E'
:=
C_E
\left(
1+\frac{c_1^2}{\mu_0^2}
\right),
\]
we obtain
\begin{equation}
\label{eq:terminal-energy}
\sum_{i=1}^2\|w^{[i]}(\cdot,\tau)\|_{L^2(\Omega)}^2
\le
C_E'
\sum_{i=1}^2\|w^{[i]}(\cdot,0)\|_{L^2(\Omega)}^2.
\end{equation}

By \eqref{eq:phi},
\[
\varphi(x,t)=\kappa t-\Psi(x),
\]
we have
\[
\varphi(x,\tau)=\kappa\tau-\Psi(x),
\]
and hence
\[
e^{-2s\varphi(x,\tau)}
=
e^{-2s\kappa\tau}e^{2s\Psi(x)}.
\]
Since
\[
\Psi(x)\leq\Psi_+
\qquad\text{for all }x\in\overline{\Omega},
\]
it follows that
\[
e^{2s\Psi(x)}
\le
e^{2s\Psi_+}.
\]
Therefore,
\[
\begin{aligned}
\int_\Omega
e^{-2s\varphi(x,\tau)}
|w^{[i]}(x,\tau)|^2\,dx
&=
e^{-2s\kappa\tau}
\int_\Omega
e^{2s\Psi(x)}
|w^{[i]}(x,\tau)|^2\,dx
\\
&\le
e^{-2s\kappa\tau}
e^{2s\Psi_+}
\int_\Omega
|w^{[i]}(x,\tau)|^2\,dx
\\
&=
e^{-2s\kappa\tau}
e^{2s\Psi_+}
\|w^{[i]}(\cdot,\tau)\|_{L^2(\Omega)}^2.
\end{aligned}
\]
Using \eqref{eq:terminal-energy}, we consequently obtain
\[
\begin{aligned}
\sum_{i=1}^2\int_\Omega
e^{-2s\varphi(x,\tau)}
|w^{[i]}(x,\tau)|^2\,dx
&\le
C_E'
e^{-2s\kappa\tau}
e^{2s\Psi_+}
\sum_{i=1}^2\|w^{[i]}(\cdot,0)\|_{L^2(\Omega)}^2.
\end{aligned}
\]

On the other hand, using \eqref{eq:phi} at $t=0$,
\[
\int_\Omega
e^{-2s\varphi(x,0)}
|w^{[i]}(x,0)|^2\,dx
=
\int_\Omega
e^{2s\Psi(x)}
|w^{[i]}(x,0)|^2\,dx
\ge
e^{2s\Psi_-}
\|w^{[i]}(\cdot,0)\|_{L^2}^2.
\]
Therefore, by the linearity of the Lebesgue integral for a finite sum,
\begin{equation}
\label{eq:terminal-absorption}
\sum_{i=1}^2\int_\Omega
e^{-2s\varphi(x,\tau)}
|w^{[i]}(x,\tau)|^2\,dx
\le
C_E'
e^{-2s[\kappa\tau-(\Psi_+-\Psi_-)]}
\int_\Omega
e^{-2s\varphi(x,0)}
\sum_{i=1}^2|w^{[i]}(x,0)|^2\,dx.
\end{equation}

By \eqref{eq:kappa-choice},
\[
\kappa\tau-(\Psi_+-\Psi_-)>0.
\]
Thus, after increasing $s_0$ if necessary,
\begin{equation}
\label{eq:terminal-absorbed-final}
Cs^2\lambda^2
\sum_{i=1}^2\int_\Omega
e^{-2s\varphi(x,\tau)}
|w^{[i]}(x,\tau)|^2\,dx
\le
\frac12
s^2\lambda^2
\int_\Omega
e^{-2s\varphi(x,0)}
\sum_{i=1}^2|w^{[i]}(x,0)|^2\,dx.
\end{equation}

\subsubsection{Step 4: Application of the Carleman estimate to the original problem}

Applying the Carleman estimate Lemma \ref{lem:carleman} to $w^{[i]}=\partial_tv^{[i]}$ for $i=1,2$ gives
\begin{equation}
\begin{aligned}
&\quad s\lambda
\int_{Q_\tau}
e^{-2s\varphi}
|\nabla w^{[i]}|^2\,dx\,dt
+
s^3\lambda^4
\int_{Q_\tau}
e^{-2s\varphi}
|w^{[i]}|^2\,dx\,dt
+
s^2\lambda^2
\int_\Omega
e^{-2s\varphi(x,0)}
|w^{[i]}(x,0)|^2\,dx
\\&
\le
C
\int_{Q_\tau}
e^{-2s\varphi}
|P_{p_1}w^{[i]}|^2\,dx\,dt
+
Cs\lambda e^{\lambda\psi_+}
\int_{\Sigma_\tau}
e^{-2s\varphi}
|\partial_\nu w^{[i]}|^2\,d\sigma\,dt
\\&\qquad+
Cs^2\lambda^2
\int_\Omega
e^{-2s\varphi(x,\tau)}
|w^{[i]}(x,\tau)|^2\,dx.
\end{aligned}
\label{eq:carleman-w}
\end{equation}

Since
\[
P_{p_1}w^{[i]}=\pi \partial_tu^{[i]}_2+\chi\partial_tu_0^{[i]}(x,t-\tau),
\]
we obtain
\begin{equation}
\begin{aligned}
&s\lambda
\int_{Q_\tau}
e^{-2s\varphi}
|\nabla w^{[i]}|^2\,dx\,dt
+
s^3\lambda^4
\int_{Q_\tau}
e^{-2s\varphi}
|w^{[i]}|^2\,dx\,dt
+
s^2\lambda^2
\int_\Omega
e^{-2s\varphi(x,0)}
|w^{[i]}(x,0)|^2\,dx
\\
&\le
C
\int_{Q_\tau}
e^{-2s\varphi}
\left|\pi \partial_tu^{[i]}_2+\chi\partial_tu_0^{[i]}(x,t-\tau)\right|^2\,dx\,dt
\\&\qquad
+
Cs\lambda e^{\lambda\psi_+}
\int_{\Sigma_\tau}
e^{-2s\varphi}
|\partial_\nu w^{[i]}|^2\,d\sigma\,dt
+
Cs^2\lambda^2
\int_\Omega
e^{-2s\varphi(x,\tau)}
|w^{[i]}(x,\tau)|^2\,dx.
\end{aligned}
\label{eq:carleman-w-source}
\end{equation}

By the terminal estimate established in \eqref{eq:terminal-absorbed-final} above,
\begin{equation}
\label{eq:terminal-absorbed-w}
Cs^2\lambda^2
\int_\Omega
e^{-2s\varphi(x,\tau)}
|w^{[i]}(x,\tau)|^2\,dx
\le
\frac12
s^2\lambda^2
\int_\Omega
e^{-2s\varphi(x,0)}
|w^{[i]}(x,0)|^2\,dx.
\end{equation}

Substituting \eqref{eq:terminal-absorbed-w} into \eqref{eq:carleman-w-source} and moving the resulting term to the left-hand side yields
\begin{equation}
\begin{aligned}
&\quad
s\lambda
\int_{Q_\tau}
e^{-2s\varphi}
|\nabla w^{[i]}|^2\,dx\,dt
+
s^3\lambda^4
\int_{Q_\tau}
e^{-2s\varphi}
|w^{[i]}|^2\,dx\,dt
+
\frac12
s^2\lambda^2
\int_\Omega
e^{-2s\varphi(x,0)}
|w^{[i]}(x,0)|^2\,dx
\\
&\le
C
\int_{Q_\tau}
e^{-2s\varphi}
\left|\pi \partial_tu^{[i]}_2+\chi\partial_tu_0^{[i]}(x,t-\tau)\right|^2\,dx\,dt
+
Cs\lambda e^{\lambda\psi_+}
\int_{\Sigma_\tau}
e^{-2s\varphi}
|\partial_\nu w^{[i]}|^2\,d\sigma\,dt.
\end{aligned}
\label{eq:carleman-w-absorbed}
\end{equation}

Finally, since the first two terms on the left-hand side of
\eqref{eq:carleman-w-absorbed} are nonnegative, we may discard
them. Thus,
\begin{equation}
\begin{aligned}
&\quad\frac12
s^2\lambda^2
\int_\Omega
e^{-2s\varphi(x,0)}
|w^{[i]}(x,0)|^2\,dx
\\&\le
C
\int_{Q_\tau}
e^{-2s\varphi}
\left|\pi \partial_tu^{[i]}_2+\chi\partial_tu_0^{[i]}(x,t-\tau)\right|^2\,dx\,dt
+
Cs\lambda e^{\lambda\psi_+}
\int_{\Sigma_\tau}
e^{-2s\varphi}
|\partial_\nu w^{[i]}|^2\,d\sigma\,dt.
\end{aligned}
\label{eq:carleman-application-half}
\end{equation}

Absorbing the factor $\frac12$ into the generic constant $C$, and taking the sum for $i=1,2$, by the linearity of the Lebesgue integral for a finite sum, we obtain
\begin{equation}
\begin{aligned}
&\quad s^2\lambda^2
\sum_{i=1}^{2}\int_\Omega
e^{-2s\varphi(x,0)}
|w^{[i]}(x,0)|^2\,dx
\\&\le
C\sum_{i=1}^{2}
\int_{Q_\tau}
e^{-2s\varphi}
\left|\pi \partial_tu^{[i]}_2+\chi\partial_tu_0^{[i]}(x,t-\tau)\right|^2\,dx\,dt
+
Cs\lambda e^{\lambda\psi_+}
\sum_{i=1}^{2}\int_{\Sigma_\tau}
e^{-2s\varphi}
|\partial_\nu w^{[i]}|^2\,d\sigma\,dt.
\end{aligned}
\label{eq:carleman-application}
\end{equation}

We now estimate the source term. By the assumed uniform bounds on $u^{[i]}_2$ and $u_0^{[i]}(x,t-\tau)$ in \eqref{eq:ut-bound},
\[
\sum_{i=1}^{2}
\left|\pi \partial_tu^{[i]}_2+\chi\partial_tu_0^{[i]}(x,t-\tau)\right|^2
\le
C_{M_1}
\left(
|\pi|^2+|\chi|^2
\right),
\]
for some constant $C_{M_1}$ depending on $M_1$. 
Hence, 
we obtain
\begin{equation}
\sum_{i=1}^{2}\int_{Q_\tau}
e^{-2s\varphi}
\left|\pi \partial_tu^{[i]}_2+\chi\partial_tu_0^{[i]}(x,t-\tau)\right|^2\,dx\,dt
\le
C_{M_1}^2
\int_\Omega
\left(
|\pi|^2+|\chi|^2
\right)
\left(
\int_0^\tau
e^{-2s\varphi(x,t)}\,dt
\right)dx.
\end{equation}

Since
\[
\varphi(x,t)=\varphi(x,0)+\kappa t,
\]
we have
\begin{equation}
\int_0^\tau e^{-2s\varphi(x,t)}\,dt
=
e^{-2s\varphi(x,0)}
\int_0^\tau e^{-2s\kappa t}\,dt
\le
\frac{1}{2s\kappa}
e^{-2s\varphi(x,0)}.
\end{equation}
Thus
\begin{equation}
\label{eq:source-weighted}
\sum_{i=1}^{2}\int_{Q_\tau}
e^{-2s\varphi}
\left|\pi \partial_tu^{[i]}_2+\chi\partial_tu_0^{[i]}(x,t-\tau)\right|^2\,dx\,dt
\le
\frac{C_{M_1}^2}{2s\kappa}
\int_\Omega
e^{-2s\varphi(x,0)}
\left(
|\pi|^2+|\chi|^2
\right)\,dx.
\end{equation}

By \eqref{eq:r-w0},
\[
|\pi(x)|^2+|\chi(x)|^2
\le
\frac{c_1^2}{\mu_0^2}
\sum_{i=1}^2
|w^{[i]}(x,0)|^2.
\]
Consequently,
\begin{equation}
\sum_{i=1}^{2}\int_{Q_\tau}
e^{-2s\varphi}
\left|\pi \partial_tu^{[i]}_2+\chi\partial_tu_0^{[i]}(x,t-\tau)\right|^2\,dx\,dt
\le
\frac{C_{M_1}^2c_1^2}
{2s\kappa\mu_0^2}
\sum_{i=1}^2\int_\Omega
e^{-2s\varphi(x,0)}
|w^{[i]}(x,0)|^2\,dx.
\label{eq:source-absorption}
\end{equation}

Choose $s$ still larger, if necessary, so that
\[
\frac{CC_{M_1}^2c_1^2}
{2s\kappa\mu_0^2}
\le
\frac12s^2\lambda^2.
\]
Equivalently, it is enough that
\[
s^3\lambda^2
\ge
\frac{CC_{M_1}^2c_1^2}
{\kappa\mu_0^2}.
\]
Then the source term is absorbed into the left-hand side of
\eqref{eq:carleman-application}, and we obtain
\begin{equation}
\label{eq:w0-weighted-final}
\sum_{i=1}^{2}\int_\Omega
e^{-2s\varphi(x,0)}
|w^{[i]}(x,0)|^2\,dx
\le
C
\frac{e^{\lambda\psi_+}}{s\lambda}
\sum_{i=1}^{2}\int_{\Sigma_\tau}
e^{-2s\varphi}
|\partial_\nu w^{[i]}|^2\,d\sigma\,dt.
\end{equation}

\subsubsection{Step 5: Removal of the weights}
Finally, we remove the exponential weights to obtain an unweighted estimate. 

Since
\[
\varphi(x,0)=-\Psi(x),
\]
we have
\[
e^{-2s\varphi(x,0)}
=
e^{2s\Psi(x)}
\ge
e^{2s\Psi_-}.
\]
Therefore, for $i=1,2$,
\begin{equation}
\label{eq:left-unweighted}
e^{2s\Psi_-}
\|w^{[i]}(\cdot,0)\|_{L^2(\Omega)}^2
\le
\int_\Omega
e^{-2s\varphi(x,0)}
|w^{[i]}(x,0)|^2\,dx.
\end{equation}

On the other hand, for $0<t<\tau$,
\[
e^{-2s\varphi(x,t)}
=
e^{2s\Psi(x)-2s\kappa t}
\le
e^{2s\Psi_+},
\]
so, for $i=1,2$,
\begin{equation}
\label{eq:right-unweighted}
\int_{\Sigma_\tau}
e^{-2s\varphi}
|\partial_\nu w^{[i]}|^2\,d\sigma\,dt
\le
e^{2s\Psi_+}
\|\partial_\nu w^{[i]}\|_{L^2(\Sigma_\tau)}^2.
\end{equation}

Combining \eqref{eq:w0-weighted-final},
\eqref{eq:left-unweighted}, and
\eqref{eq:right-unweighted}, and summing over $i=1,2$, by the linearity of the Lebesgue integral for a finite sum, we obtain
\begin{equation}
\label{eq:w0-boundary}
\sum_{i=1}^{2}\|w^{[i]}(\cdot,0)\|_{L^2(\Omega)}
\le
C_f
\sum_{i=1}^{2}\|\partial_\nu w^{[i]}\|_{L^2(\Sigma_\tau)},
\end{equation}
where
\[
C_f
=
C^{1/2}
(s\lambda)^{-1/2}
e^{\lambda\psi_+/2}
e^{s(\Psi_+-\Psi_-)}.
\]
The parameters $\lambda$, $s$, and $\kappa$ are fixed entirely in terms of the a priori data, so $C_f$ is a finite a priori constant.

For $i=1,2$, since
\[
w^{[i]}=\partial_\nu v^{[i]},
\]
we have, by the regularity of $u_j^{[i]}$, $i,j=1,2$,
\[
\partial_\nu w^{[i]}
=
\partial_\nu \partial_t v^{[i]}
=
\partial_t\partial_\nu v^{[i]}.
\]
Because
\[
v^{[i]}=u_1^{[i]}-u_2^{[i]},
\]
it follows that
\[
\partial_\nu w^{[i]}
=
\partial_t\partial_\nu u_1^{[i]}
-
\partial_t\partial_\nu u_2^{[i]}.
\]
Thus
\begin{equation}
\label{eq:w0-observation-final}
\sum_{i=1}^{2}\|w^{[i]}(\cdot,0)\|_{L^2(\Omega)}
\le
C_f
\sum_{i=1}^{2}\left\|
\partial_t\partial_\nu u_1^{[i]}
-
\partial_t\partial_\nu u_2^{[i]}
\right\|_{L^2(\Sigma_\tau)}.
\end{equation}

\subsubsection{Step 6: Completion of the proof}

From \eqref{eq:r-w0},
\[|\pi(x)|^2+|\chi(x)|^2
\le
\frac{c_1^2}{\mu_0^2}
\sum_{i=1}^2
|w^{[i]}(x,0)|^2.
\]
Taking the integral over all $x\in\Omega$ and applying \eqref{eq:w0-observation-final}, we have
\begin{equation}\label{eq:pi-chi-stability-final}
\|\pi\|_{L^2(\Omega)}+\|\chi\|_{L^2(\Omega)}
\le
\frac{c_1C_f}{\mu_0}
\sum_{i=1}^2\left\|
\partial_t\partial_\nu u_1
-
\partial_t\partial_\nu u_2
\right\|_{L^2(\Sigma_\tau)}.
\end{equation}
Since $\pi=p_1-p_2$ and $\chi=q_1-q_2$, this is exactly \eqref{eq:final-stability}.

Finally, suppose that
\[
\partial_\nu u_1^{[i]}
=
\partial_\nu u_2^{[i]}
\qquad\text{on }\Sigma_\tau,
\qquad i=1,2.
\]
Then,
\[
\partial_\nu v^{[i]}=0
\qquad\text{on }\Sigma_\tau,
\]
and hence, by the time regularity,
\[
\partial_\nu w^{[i]}
=
\partial_t\partial_\nu v^{[i]}
=0
\qquad\text{on }\Sigma_\tau.
\]
It follows from \eqref{eq:pi-chi-stability-final} that
\[
\|\pi\|_{L^2(\Omega)}^2
+
\|\chi\|_{L^2(\Omega)}^2
=0.
\]
Thus
\[
\pi=\chi=0
\qquad\text{a.e. in }\Omega.
\]
Therefore, we conclude that
\[
p_1=p_2,
\qquad
q_1=q_2
\qquad\text{a.e. in }\Omega.
\]
The proof is complete.
\qed

\begin{remark}[The delayed diffusion coefficient $d_2$]
The coefficient $d_2$ may be arbitrary. Indeed, the difference of the two solutions satisfies \eqref{eq:v-initial-layer}. 
Thus the delayed diffusion term disappears identically from the equation on the entire interval on which the Carleman estimate is applied. The proof therefore requires no condition on $d_2$.
\end{remark}

\begin{remark}[No smallness assumption on the delay]
The only restriction on the delay is
\(0<\tau<T.\)
The Carleman parameter $\kappa$ is chosen after $\tau$ so that
\[
\kappa\tau>\Psi_+-\Psi_-.
\]
Since $\kappa$ is a free Carleman parameter, this condition can always be fulfilled for every positive $\tau$.
Thus the proof contains no hidden assumption of the form \(\tau<\tau_0.\)
\end{remark}

\begin{remark}[Why $p,q$ must be independent of time]
The proof differentiates $p,q$ in \eqref{eq:w-equation}
If instead $p=p(x,t)$ and $q=q(x,t)$, additional terms
would appear. The present proof therefore does not establish the theorem for time-dependent zeroth-order coefficients.

\end{remark}

\begin{remark}[The nondegeneracy condition]
The identity
\[
c_1w^{[i]}(x,0)
=
r(x)u_2^{[i]}(x,0)+su_0^{[i]}(x,-\tau)
\]
is the mechanism by which the unknown coefficient difference is converted into an initial trace. Therefore the uniform nondegeneracy condition \eqref{eq:pqnondeg} is essential for the stated global Lipschitz estimate.

In the case where $q$ is known and we only require the stability for $p$, this condition reduces to 
\[
u_0(\cdot,0)\in W^{2,\infty}(\Omega)
\]
and there exists $\kappa_0>0$ such that
\begin{equation}
\label{eq:nondegenerate}
|u_0(x,0)|\ge\kappa_0
\qquad
\text{for every }x\in\overline\Omega.
\end{equation}

If the initial value vanishes on a subset of $\Omega$, the same argument can at most determine $p_1-p_2$ and $q_1-q_2$ on the subset where $u_0(\cdot,0)$ is bounded away from zero.

\end{remark}

\begin{remark}
The proof does not require a Carleman estimate for the delayed operator on the full cylinder $Q_T$. The delay enters only through the forward problem and disappears from the difference equation on $Q_\tau$.

\end{remark}

\begin{remark}[Remark on the necessity of the initial layer]

The restriction to the time interval $(0,\tau)$ -- the initial layer -- is not an arbitrary technical choice but a structural necessity for the proof. 

The inverse stability argument relies on two key facts that hold only on $(0,\tau)$:

\begin{enumerate}
\item For any $0<t<\tau$, the shifted time $t-\tau$ lies in the history interval $[-\tau,0]$. Because the two solutions share identical history data $u_0$, their difference $v$ vanishes identically on this interval. Consequently, the delayed terms
\[
d_2\Delta v(x,t-\tau)
\quad\text{and}\quad
q(x)v(x,t-\tau)
\]
are exactly zero in the difference equation.

\item This exact cancellation is what allows the reduction to a non-delayed parabolic equation
\begin{equation}
c_1v_t-d_1\Delta v-p_1v=ru_2+su_0 \qquad\text{in }Q_\tau.
\label{eq:non-delayed-reduction}
\end{equation}
The subsequent Carleman estimate is designed precisely for this non-delayed operator.
\end{enumerate}

If one attempted to work on any interval beyond $\tau$, the delayed terms would no longer vanish pointwise (since $t-\tau>0$), and the difference equation would retain its full delayed structure:
\begin{equation}
c_1v_t-d_1\Delta v-d_2\Delta v(x,t-\tau)-p_1v-qv(x,t-\tau)=ru_2+su_0.
\end{equation}
The present Carleman estimate, which differentiates in time and relies on $p$ being time-independent, would not apply to such a delayed equation without significant additional assumptions (e.g., smallness of $d_2$ or a different Carleman weight adapted to the delay).

Thus, the initial layer is necessary precisely because it is the only window where the delay is ``invisible" in the difference equation, allowing the inverse problem to be recast as a classical (non-delayed) inverse source problem for the initial trace $w(\cdot,0)$.




\end{remark}

\section{Discussion and Future Directions}

This work addresses an inverse problem for a broad class of linear reaction--diffusion equations incorporating multiple delayed mechanisms: delayed diffusion, delayed time derivatives, and delayed reaction terms. We aim to simultaneously recover an unknown delay parameter $\tau>0$ and spatially varying reaction coefficients using only boundary flux measurements --- a setting motivated by the widespread inaccessibility of internal state observations in biological systems.

Our results carry key implications for biological modeling and monitoring. The ability to infer delay time from boundary flux alone is particularly valuable for ecological and epidemiological applications. For instance, species maturation delays can be estimated from individual fluxes across reserve boundaries without invasive internal sampling, and pathogen incubation periods can be recovered from case time series at monitoring region boundaries.

The initial-layer coefficient recovery also has practical relevance. The interval $(0,\tau)$, corresponding to the period before delayed feedback from the system's own dynamics takes effect, constitutes an optimal window for parameter estimation. In practice, early-time boundary measurements --- before the delay mechanism is fully engaged --- contain sufficient information to determine local growth and interaction rates. This finding can guide experimental design: concentrating measurements in the early phase maximizes the information gain about unknown parameters.

The Lipschitz stability estimate further reinforces the practical viability of our approach. It quantifies the sensitivity of reconstructed coefficients to boundary data perturbations, theoretically guaranteeing that small measurement errors do not lead to disproportionately large errors in recovered parameters --- a critical property for inversion methods applied to noisy real-world data.

\subsection{Concluding Remarks}

In summary, we have developed a complete analytical framework for solving inverse problems in delayed reaction--diffusion equations using only boundary flux measurements. The combination of singularity detection, method-of-steps reduction, and Carleman estimates effectively identifies both the delay parameter and spatially varying reaction coefficients, with a quantitative stability result established in a principal special case. While limitations remain regarding nonlinearities, unknown diffusion parameters, and multiple delays, this work lays a solid foundation for further theoretical and numerical developments.

More broadly, this study helps bridge the gap between mathematical inverse theory and practical biological monitoring. By demonstrating the feasibility of non-invasive parameter identification in principle, we hope to stimulate research into efficient algorithms and experimental protocols that translate these theoretical insights into actionable tools for conservation, epidemiology, and systems biology. The interface between delayed dynamics and inverse problems remains rich with open questions, and the approach presented here offers a useful starting point for exploring this promising area.

\medskip

\noindent\textbf{Acknowledgment.} 
	The work of M. Ding is supported by the National Natural Science Foundation of China (No. 1250012409) and Fundamental Research Funds for the Central Universities under Grant D5000240301. The work of H. Liu is partially supported by the Hong Kong RGC General Research Funds (No. 11303125, 11304224 and 11311122). The work of C. W. K. Lo is supported by the National Natural Science Foundation of China (No. 12501660), the Guangdong Province ``Pearl River" Talent Recruitment Program (Top youth talent) (No. 2024QN11X268), the Shenzhen Science and Technology Program (No. QNXMC20250701094200001) and the Guangdong Key Laboratory of Applied Mathematics and Artificial Intelligence.

\noindent\textbf{Conflict of interest statement.} 
	The authors declare that they have no conflict of interests that could have appeared to influence the work reported in this paper.

\noindent\textbf{Data availability statement.} 
We do not analyse or generate any datasets, because our work proceeds within a theoretical and mathematical approach.

\bibliographystyle{plain}
\bibliography{references}
\end{document}